\documentclass[11pt]{article}
\pdfoutput=1
\usepackage[T2A,T1]{fontenc}
\usepackage[utf8]{inputenc}
\usepackage[russian,english]{babel}

\usepackage{graphicx,amsmath,amssymb,amsthm}
\usepackage[outdir=./fig/]{epstopdf}
\usepackage{tikz, verbatim, cleveref}
\usepackage{pifont, url, romanbar,subfig,listings}
\usepackage[labelfont=bf]{caption}
\usepackage[top=50pt, bottom=50pt, left=50pt, right=50pt]{geometry}
\usepackage{autonum}

\allowdisplaybreaks

\newcommand{\R}{{\operatorname{right}}}
\renewcommand{\L}{{\operatorname{left}}}
\renewcommand{\O}{{\operatorname{outer}}}

\newtheorem{lemma}{Lemma}[section]
\newtheorem{proposition}[lemma]{Proposition}
\newtheorem{remark}[lemma]{Remark}
\usetikzlibrary{patterns,arrows,decorations.markings}

\numberwithin{equation}{section}
\newcommand{\Ri}{\Romanbar{1}}
\newcommand{\Ro}{\Romanbar{2}}
\newcommand{\Rright}{\Romanbar{3}}
\newcommand{\Rleft}{\Romanbar{4}}

\addto\captionsenglish{
    \crefname{figure}{Figure}{Figures}
    \Crefname{figure}{Figure}{Figures}
    \crefname{table}{Table}{Tables}
    \Crefname{table}{Table}{Tables}
    \crefname{section}{\S}{\S}
    \Crefname{section}{\S}{\S}
    \crefname{equation}{}{}
    \Crefname{equation}{}{}
    \crefname{remark}{Remark}{Remarks}
    \Crefname{remark}{Remark}{Remarks}
}

\title{Construction and implementation of asymptotic expansions for Laguerre--type orthogonal polynomials}
\author{
Daan Huybrechs \\ daan.huybrechs@cs.kuleuven.be\\
Department of Computer Science\\
KU Leuven, Belgium
\and
Peter Opsomer (corresponding author) \\ peter.opsomer@cs.kuleuven.be \\
Department of Computer Science\\
KU Leuven, Belgium
}

\begin{document}

\maketitle

\begin{abstract}
%Alfredo/ In a previous paper we detailed the construction and implementation of high-order asymptotic expansions for Jacobi and more general Jacobi-type polynomials, based on their asymptotic analysis available in mathematical literature. In the present paper we repeat this effort for Laguerre and Laguerre-type polynomials. These are orthogonal on the interval $[0,\infty)$ with respect to a weight function of the form
Laguerre and Laguerre-type polynomials are orthogonal polynomials on the interval $[0,\infty)$ with respect to a weight function of the form
\begin{equation}
  w(x) = x^{\alpha} e^{-Q(x)}, \quad Q(x) = \sum_{k=0}^m q_k x^k, \quad \alpha > -1, \quad q_m > 0. \nonumber
\end{equation}
The classical Laguerre polynomials correspond to $Q(x)=x$. The computation of higher-order terms of the asymptotic expansions of these polynomials for large degree becomes quite complicated, and a full description seems to be lacking in literature. However, this information is implicitly available in the work of Vanlessen \cite{Vanlessen}, based on a non-linear steepest descent analysis of an associated so-called Riemann--Hilbert problem. We will extend this work and show how to efficiently compute an arbitrary number of higher-order terms in the asymptotic expansions of Laguerre and Laguerre-type polynomials. This effort is similar to the case of Jacobi and Jacobi-type polynomials in a previous paper. %Alfredo/ added here
We supply an implementation with explicit expansions in four different regions of the complex plane. These expansions can also be extended to Hermite-type weights of the form $\exp(-\sum_{k=0}^m q_k x^{2k})$ on $(-\infty,\infty)$, and to general non-polynomial functions $Q(x)$ using contour integrals. The expansions may be used, e.g., to compute Gauss-Laguerre quadrature rules in a lower computational complexity than based on the recurrence relation, and with improved accuracy for large degree. They are also of interest in random matrix theory.
\end{abstract}

\section{Introduction}\label{Sintro}

We determine asymptotic approximations as $n \rightarrow \infty$ of the orthonormal polynomials $p_n(x)$ on $[0,\infty)$ with positive leading coefficient, with the weight function 
\[
  w(x) = x^{\alpha} e^{-Q(x)}, \quad Q(x) = \sum_{k=0}^m q_k x^k, \quad \alpha > -1, \quad q_m > 0.
\]
The classical Laguerre polynomials corresponds to $Q(x)=x$, but we aim to provide formulas and results for general functions $Q(x)$. The choice $Q(x)=x^m$ corresponds to so-called Freud-type polynomials~\cite{Levin2001weights}.

The procedure in \cite{Vanlessen} gives four types of asymptotic expansions: (\Ri) inner asymptotics for $x$ near the bulk of the zeros but away from the extreme zeros, (\Ro) outer asymptotics valid for $x$ away from the zeros of $p_n(x)$, (\Rright) boundary asymptotics near the so-called soft edge valid for $x$ near the largest zeros and (\Rleft) boundary asymptotics near the so-called hard edge for $x$ near $0$. We also provide asymptotic expansions for associated quantities such as leading term coefficients as well as recurrence coefficients $a_n$ and $b_n$ of the three term recurrence relation
\begin{equation}\label{recur}
	b_n p_{n+1}(x) = (x-a_n)p_{n}(x) - b_{n-1}p_{n-1}(x).
\end{equation}

The methodology of \cite{Vanlessen} is based on the non-linear steepest descent method by Deift and Zhou \cite{dzsd} for a $2 \times 2$ Riemann--Hilbert problem that is generically associated with orthogonal polynomials by Fokas, Its and Kitaev \cite{fik}. This is further detailed in \cref{SRH}. %We provide an algorithm to obtain an arbitrary number of terms in the expansions, where we set up series expansions using many convolutions that follow the chain of transformations and their inverses in this steepest descent method. While doing this, we keep computational efficiency in mind, as well as the use of the correct branch cuts in the complex plane.%The methodology of \cite{Vanlessen} is based on the non-linear steepest descent method by Deift and Zhou \cite{dzsd} for a $2 \times 2$ Riemann--Hilbert problem that is generically associated with orthogonal polynomials by Fokas, Its and Kitaev \cite{fik}. The solution of this problem is a matrix-valued function $Y(z)$ that satisfies certain jump relations across contours in the complex plane. 
The general strategy is to apply a sequence of transformations $Y(z)\mapsto T(z) \mapsto S(z) \mapsto R(z)$, such that  the final matrix-valued function $R(z)$ is asymptotically close to the identity matrix as $n$ or $z$ tends to $\infty$. %Alfredo/ added
The asymptotic result for $Y(z)$, and subsequently for the polynomials, is obtained by inverting these transformations. The transformations involve a normalization of behaviour at infinity, the so-called `opening of a lens' %Alfredo/
around the interval of orthogonality, and the introduction of local parametrices in disks around special points like endpoints, which are matched to global parametrices elsewhere in the complex plane. These transformations split $\mathbb{C}$ into different regions, where different formulas for the asymptotics are valid.

In our case of Laguerre-type polynomials, one also first needs an $n$-dependent rescaling of the $x$ axis using the so-called MRS numbers (defined further on in this paper). After this step the roots of the rescaled polynomials accumulate in a fixed and finite interval. We provide an algorithm to obtain an arbitrary number of terms in the expansions, where we set up series expansions using many convolutions that follow the chain of transformations and their inverses in this steepest descent method. While doing this, we keep computational efficiency in mind, as well as the use of the correct branch cuts in the complex plane.%We provide an algorithm to obtain an arbitrary number of terms in the expansions, where we set up series expansions using many convolutions that follow the chain of transformations and their inverses. In all of this, one should also keep efficiency in mind, as well as the use of the correct branch cuts in the complex plane.

The strategy outlined above and in \cref{SRH} was also followed in our earlier article about asymptotic expansions of Jacobi--type polynomials \cite{jacobi}, which was based on the mathematical analysis of Kuijlaars et al in \cite{KMcLVAV}. Here, we base our results on the analysis in the work of Vanlessen \cite{Vanlessen}. The main differences in this paper compared to \cite{jacobi} are the following:
\begin{itemize}
\item The analysis of Laguerre-type polynomials on the halfline $[0,\infty)$ has an extra step that involves a rescaling via the MRS numbers $\beta_n$ (see \cref{s:MRS}). This leads to fractional powers of $n$ in the expansions.
\item There is also a new behaviour near the largest zero (often referred to as a soft edge), captured by the Airy function. This leads to higher order poles in the derivations and thus also to longer formulas for the higher order terms in the expansions. The behaviour near the hard edge at $x=0$ involves Bessel functions, like in the Jacobi case near the endpoints $\pm 1$.
\item We obtain more explicit results for polynomials $Q(x)$. We will frequently distinguish in this paper between three cases: monomial $Q$, general polynomial $Q$ and a more general analytic function $Q$.
\end{itemize} %Alfredo/ ook

Asymptotic expansions can be useful in computations for several reasons. The use of the recurrence relation \eqref{recur} in applications involving Laguerre polynomials results in accumulating roundoff errors %Alfredo/ inderdaad belangrijk bij nulpunten maar misschien niet echt nodig om hier te vermelden?
and a computation time that is linear in $n$. In contrast, asymptotic expansions become increasingly accurate as the degree $n$ becomes large, and the computing time is essentially independent of $n$.  Another motivation is the partition function for Laguerre ensembles of random matrices as mentioned in \cite{Vanlessen} and studied in \cite{ZhaoPartition}. This gives the eigenvalue distribution of products of random matrices of a certain type, which could for example arise in stochastic processes (Markov chains) or quantum mechanics.%Alfredo/ Markov chains toegevoegd

In this reference, some leading order terms are given explicitly, and we detail the derivation of higher-order terms. The analysis only requires elementary numerical techniques: in particular, there is no need for the evaluation of special functions besides the Airy and Bessel functions. The formulas are implemented in {\sc Sage} and {\sc Matlab} and are available on the software web page of our research group \cite{ninesLag}.

The expressions are also implemented in the Chebfun package in Matlab for computing with functions (see \cite[lagpts.m]{chebfun}) and into the Julia package FastGaussQuadrature \cite[gausslaguerre.jl]{FastGaussQuadr}, in both cases for the purpose of constructing Gauss-Laguerre quadrature rules with a high number of points in linear complexity. This approach is comparable to a number of modern numerical methods for other types of Gaussian quadrature, that are often based on asymptotics and that lead to a linear complexity as well~\cite{glaser,bogaert,HT,BogaertIterationFree}. A recent paper \cite{bremer} achieves competitive performance in a general way via nonoscillatory phase functions. Potential improvements to our code in light of this result and a more thorough discussion of the contributions to the computation of Gaussian quadrature rules are future research topics. We do remark here that the construction of Gaussian quadrature rules requires the derivative of the associated orthonormal Laguerre 
polynomial with positive leading coefficient. This can be obtained from our expansions and the identity $d p_n^{(\alpha)}(x)/dx = \sqrt{n}p_{n-1}^{(\alpha+1)}(x)$, derived from \cite[18.9.23]{DLMF}. % <- cite toegevoegd
Although technical, the expansions can readily be differentiated for general $Q(x)$. This paper affirmatively answers the questions raised in the conclusions of \cite{TTOGauss}, namely whether the RH approach can be applied to the fast computation of quadrature rules with generalized Laguerre weights (for the Jacobi case, see \cite{jacobi}), and whether higher order asymptotic expansions can be computed effectively.

As mentioned, the standard Laguerre polynomials correspond to $Q(x) = x$, or equivalently $q_k \equiv 0$, $\forall k \neq 1$, and $q_1=1$. For this case, asymptotic expansions are given in \cite{alfLag} with explicit expressions for the first terms. We refer the reader to \cite{temme1990,temme2004} and references therein for more results on asymptotics for the standard Laguerre polynomials. A recent scheme for the numerical evaluation of Laguerre polynomials of any degree is described in \cite{gil2016}.

As an example, the type of expansions in this paper have the following form. For the monic Laguerre polynomial $\pi_n$ of degree $n$, we obtain:
\begin{align} 
	\pi_n(x) = p_n(4nz) &= \frac{(4n)^{n} e^{n (2z-1-2\log(2))} }{z^{1/4}(1-z)^{1/4} z^{\alpha/2} } \begin{pmatrix} 1 \\ 0 \end{pmatrix}^T R^{\O}(z)  \\ %	\pi_n(x) = p_n(4nz) &= \frac{(4n)^{n} e^{n (V_n(z)/2+l_n/2)} }{z^{1/4}(1-z)^{1/4} z^{\alpha/2} } \begin{pmatrix} 1 \\ 0 \end{pmatrix}^T R^{\O}(z)  \\
	&  \begin{pmatrix} 2^{-\alpha} \cos(\arccos(2z-1)[1/2+\alpha/2] - n[2\sqrt(z)\sqrt{1-z}-2\arccos(\sqrt{z}) ] -\pi  /4)  \\  -i 2^{\alpha} \cos(\arccos(2z-1)[\alpha/2-1/2] - n[2\sqrt(z)\sqrt{1-z}-2\arccos(\sqrt{z}) ] -\pi /4) \end{pmatrix}.  \nonumber%&  \begin{pmatrix} 2^{-\alpha} \cos(\arccos(2z-1)[1/2+\alpha/2] + n\xi_n(z)/i -\pi  /4)  \\  -i 2^{\alpha} \cos(\arccos(2z-1)[\alpha/2-1/2] + n\xi_n(z)/i -\pi /4) \end{pmatrix}.  \nonumber
\end{align}%Alfredo/ O(n^{-1}z) nu weg en evt Szego controleren %Alfredo/ Lijkt niet direct overeen te komen met Mehler-Heine Szego 8.1.8: lim_{n->infty} n^{-alpha/2} L_n^alpha(z/n) = z^(-alpha/2)J_alpha(2 z^{1/2}), want die is voor z/n-> 0 (dus Bessel region) dus misschien beter hier niets over zeggen? MIsschien bij sectie 4.4 Left disk IV?
This is expression \eqref{EpiInt} of the paper, specified to the standard associated Laguerre weight $w(x)=x^{\alpha}e^{-x}$. It is valid for $x \in (0,4n)$, where $x$ is related to $z$ through the MRS number $\beta_n = 4n$, i.e., $x=4nz$. The expansion itself follows from substituting the expansion of the $2 \times 2$ matrix function $R^{\O}(z)$. The leading order term is obtained from the identity matrix $R^{\O}(z)=I$, and further terms are listed explicitly in Appendix \ref{APP:explicit}. We provide formulas and their implementation for an arbitrary number of terms in the asymptotic expansions, and compute up to $50$ terms in $32$ seconds on the architecture mentioned in \cref{SnumerGen}. %in Appendix \ref{APP:explicit}. In our Matlab implementation it is feasible to compute up to $20$ terms.

These formulas can also be applied to obtain asymptotic expansions of orthogonal polynomials with Hermite-type weights of the form $\exp(-\sum_{k=0}^m q_k x^{2k})$ on $(-\infty,\infty)$. In \cref{Sherm}, we show that they can be given in terms of asymptotics of Laguerre-type polynomials with $\alpha = \pm 1/2$, evaluated in $x^2$.

We aim for a general non-polynomial weight function $Q(x)$, though our results in this case are thus far not rigorously valid. In particular, we do not provide estimates for the remainder term. We do provide numerical indications that the expansions converge at the expected rate for increasing $n$. Inspired by the requirements for the Jacobi case \cite{jacobi} and the technical conditions on $Q(x)$ in \cite[\S 1]{Levin2001weights}, we conjecture that the expansions in this paper are valid as long as $Q(x)$ is analytic within the contours defined further on and $Q(x)$ grows faster than powers of $\log(x)$ for $x\rightarrow \infty$.

The structure of the paper is as follows. In \cref{Sasy}, we connect the Riemann-Hilbert problem for orthogonal polynomials that is analyzed in \cite{Vanlessen} with the expansion of a $2 \times 2$ matrix-valued function $R$ and introduce some notation. %In \S\ref{Sasy} we provide the asymptotic expansions in terms of the expansion of a $2 \times 2$ matrix-valued function $R$. This function arises in the Riemann-Hilbert problem for orthogonal polynomials that is analyzed in \cite{Vanlessen}. 
We detail the Mhaskar-Rakhmanov-Saff (MRS) numbers $\beta_n$ and their asymptotic expansions for large $n$ in \S\ref{s:MRS}. The formulas for the asymptotic expansions of the polynomials in the different regions of the complex plane are stated in \S\ref{s:asymptotics}. We explain the computation of higher order terms of $R$ in \S\ref{S:higherorderterms} and provide a non-recursive definition for $R$. Details on obtaining explicit expressions for higher order terms are provided in \S\ref{SexplL}. We conclude the paper with a number of examples and numerical results in \S\ref{Snum}.

\section{Asymptotic expansions for Laguerre--type polynomials}\label{Sasy}

The largest root of a Laguerre-type polynomial $p_n(x)$ grows with the degree $n$. 
For example, it asymptotically behaves as $4n+2\alpha+2 + 2^{2/3}a_1(4n+2\alpha+2)^{1/3}$ for the standard associated Laguerre polynomials, with $a_1$ the (negative) zero of the Airy function closest to zero, see \cite[(18.16.14)]{DLMF,Olver:2010:NHMF} and \cite[(6.32.4)]{Szego}. %Alfredo/ Toegevoegd
The first step in the description of the asymptotics is to rescale the polynomials, such that the support of the zero-counting measure maps to the interval $[0,1]$. The scaling is linear but $n$-dependent and given by
\begin{equation}\label{eq:scaling}
x = \beta_n z,
\end{equation}
where $\beta_n$ is the Mhaskar-Rakhmanov-Saff (MRS) number~\cite{Levin2001weights} defined further on in \eqref{EbetaInt}.

There is a distinction between several regions in the complex $z$-plane, shown in Figure \ref{Fregions}:
\begin{itemize}
 \item a complex neighbourhood of the interval $(0,1)$ excluding the endpoints, subsequently called the `lens' (region \Ri)
 \item two disks around the endpoints $1$ and $0$, called the right and left disk (regions \Rright~and \Rleft)
 \item and the remainder of the complex plane, the `outer region' (region \Ro).
\end{itemize}

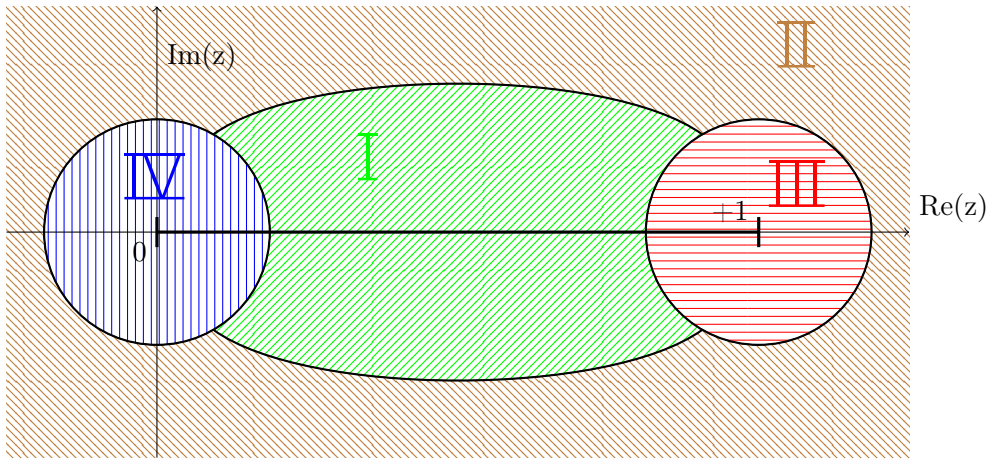
\begin{figure}[t]
\begin{center}
\begin{tikzpicture}
	\fill[pattern=north west lines, pattern color =brown] (0,4) rectangle (12, -2);
	\fill[white] (2.75,2.32) rectangle (9.25, -0.32);
	\draw (2,1) [fill=white] circle [radius=1.5];
	\draw (10,1) [fill=white] circle [radius=1.5];
	\draw[fill=white] (2.75, 2.3) .. controls(4,3.2) and (8, 3.2) .. (9.25, 2.3);
	\draw[fill=white] (2.75, -0.3) .. controls(4,-1.2) and (8, -1.2) .. (9.25, -0.3);

	\fill[pattern=north east lines, pattern color =green] (2.75,2.3) rectangle (9.25, -0.3);
	\draw (2,1) [fill=white] circle [radius=1.5];
	\draw (10,1) [fill=white] circle [radius=1.5];
	\draw (10,1) [thick,pattern=horizontal lines, pattern color =red] circle [radius=1.5];
	\draw (2,1) [thick,pattern=vertical lines, pattern color =blue] circle [radius=1.5];
	\draw[thick,pattern=north east lines, pattern color =green] (2.75, 2.3) .. controls(4,3.2) and (8, 3.2) .. (9.25, 2.3);
	\draw[thick,pattern=north east lines, pattern color =green] (2.75, -0.3) .. controls(4,-1.2) and (8, -1.2) .. (9.25, -0.3);

	\draw[very thick] (2,1)--(10,1);
	\draw[->] (2,-2)--(2,4);
	\draw (2,3) node[above right] {Im(z)};
	\draw[->] (0,1)--(12,1);
	\draw (12,1) node[above right] {Re(z)};
	\draw (2,1) node[below left] {0};
	\draw (10,1) node[above left] {+1};
	\draw[very thick] (2,0.8)--(2,1.2);
	\draw[very thick] (10,0.8)--(10,1.2);
	\draw[green] (4.8,2) node {\Huge{\Ri } };
	\draw[brown] (10.5,3.5) node {\Huge{\Ro } };
	\draw (10,1.2) node[red,above right] {\Huge{\Rright } };
	\draw (2.5,1.3) node[above left,blue] {\Huge{\Rleft } };
\end{tikzpicture} 
\end{center}
\caption{Regions of the complex plane in which the polynomials have different asymptotic expansions, after rescaling the support of the zero-counting measure to the interval $[0,1]$: the lens (\Ri, a complex neighbourhood of the interval $(0,1)$ excluding the endpoints), the outer region (\Ro, the remainder of the complex plane) and the right and left disks (\Rright~and \Rleft, two disks around the endpoints $0$ and $1$).}
\label{Fregions}
\end{figure}

\subsection{Riemann--Hilbert formulation and steepest descent analysis} \label{SRH}

In this section, we briefly summarize the main features of the derivation in \cite{Vanlessen}. The approach is based on the Riemann--Hilbert formulation for orthogonal polynomials \cite{fik}: we seek a $2\times 2$ complex matrix-valued function $Y(z)$ that satisfies the following Riemann--Hilbert problem (RHP), cf. \cite[\S 3]{Vanlessen}:
\begin{enumerate}
	\item[(a)] $Y$ : $\mathbb{C}\setminus[0,\infty) \rightarrow \mathbb{C}^{2 \times 2}$ is analytic.
	\item[(b)] $Y(z)$ has continuous boundary values $Y_{\pm}(x)$, when going from the upper half-plane through the interval $(0,\infty)$ to the lower half-plane, respectively. These boundary values are related via a jump matrix:
	\begin{equation}
		Y_+(x) = Y_-(x) \begin{pmatrix}1 & x^\alpha e^{-Q(x)} \\ 0 & 1\end{pmatrix}, \quad x \in (0,\infty).\nonumber
	\end{equation}
 	\item[(c)] $Y(z)\begin{pmatrix} z^{-n} & 0\\ 0 & z^{n}\end{pmatrix} = I+\mathcal{O}\left(\frac{1}{z}\right), \qquad z\to\infty$.
	\item[(d)] The behaviour as $z \rightarrow 0$ is also specified, see \cite[(3.3)]{Vanlessen}. 
\end{enumerate}

It is proved in \cite{fik,KuijLect}, that the unique solution of this Riemann--Hilbert problem is 
\begin{equation}
	Y(z) = \begin{pmatrix} p_n(z)/\gamma_n & \frac{1}{2\pi i\gamma_n}\int_{0}^\infty \frac{p_n(x) w(x)}{x-z} dx \\ -2\pi i \gamma_{n-1} p_{n-1}(z) & -\gamma_{n-1}\int_{0}^\infty \frac{p_{n-1}(x) w(x)}{x-z} dx \end{pmatrix}. \nonumber
\end{equation}
Here, the $Y_{11}$ entry is the monic orthogonal polynomial. The $Y_{21}$ entry relates to the polynomial of degree $n-1$, while the second column contains the Cauchy transforms of both of these polynomials. Note that the weight function of the orthogonal polynomials enters through the jump condition in (b). 

In order to obtain the large $n$ asymptotic behavior of $p_n(x)$, the Riemann--Hilbert formulation is combined with the Deift--Zhou steepest descent method for Riemann--Hilbert problems \cite{DZ,dzsd}. In this case, the steepest descent analysis presented in \cite{Vanlessen} consists of the following sequence of (explicit and invertible) transformations:
\begin{equation}
	Y(z)\mapsto T(z) \mapsto S(z) \mapsto R(z). \nonumber
\end{equation}
These steps have a well-defined interpretation:
\begin{itemize} 
 \item The first step is a normalization at infinity, such that
\begin{equation}
	T(z) = I+\mathcal{O}(1/z), \qquad z\rightarrow \infty. \label{ETinf} 
\end{equation}
This step comes at the cost of introducing rapidly oscillating entries in the new jump matrix for the Riemann-Hilbert problem for $T$.

\item The second step is the opening of the so-called lens around $[0,1]$: it factorizes the previous jump matrix such that $S(z) = T(z)$ outside of the lens in Figure \ref{Fregions}, while $S(z)$ is exponentially close to $T(z)$ in $n$ in the upper and lower part of the lens. The shape of the lens is such that the oscillating entries on the diagonal in the jump matrix are transformed into exponentially decaying off-diagonal entries.

\item Finally, the last transformation $S(z) \mapsto R(z)$ gives rise to the disks in \cref{Fregions} and is defined as
	\begin{equation}
		R(z)=S(z)
		\begin{cases}
			P_n(z)^{-1}, \qquad \text{for z close to }1, \\
			\tilde{P}_n(z)^{-1}, \qquad \text{for z close to }0, \\
			P^{(\infty)}(z)^{-1}, \qquad \text{elsewhere.}
		\end{cases} \label{EjumpR}
	\end{equation}
	Here, a global parametrix $P^{(\infty)}(z)$ is introduced and constructed using the Szeg\H{o} function $z^{\alpha/2}\varphi(z)^{-\alpha/2}$%$z^{\alpha/2}\phi(z)^{-\alpha/2}$
 (to be defined below). $P_n(z)$ and $\tilde{P}_n(z)$ are local parametrices that follow from a rather involved local analysis around the endpoints. We omit the details, but we note that $\tilde{P}_n(z)$ is given explicitly in terms of standard Bessel and Hankel functions and their derivatives. The precise choice of these functions is made in such a way that $\tilde{P}_n(z)$ satisfies a matching condition with the global parametrix on the boundary of the left disk, namely 
	\begin{equation}
		\tilde{P}_n(z)\left[P^{(\infty)}(z)\right]^{-1} = I + \Delta^{\L}(z), \nonumber
	\end{equation}
	where $\Delta^{\L}(z)$ will be defined in \cref{ss:localjump} as the jump matrix for $R(z)$. A similar construction yields explicit expressions for $P_n(z)$ in terms of the Airy function and its derivative. Since we know $P_n(z)$, $\tilde{P}_n(z)$ and $P^{(\infty)}(z)$ explicitly, we can determine the asymptotic expansion of $\Delta^{\R/\L}(z)$ in a closed formula.
\end{itemize}

% We do this because the problem for $R(z)$ is somewhat `easier' to solve since $R(z)$ is close to $I$ for large $n$ and/or $z$, which eventually leads to the procedure as explained in \S \ref{Shot}.

The key idea is that the Riemann--Hilbert problem for $R(z)$ can be solved explicitly in an asymptotic sense for large $n$: it can be deduced that the matrix $R(z)$ is itself close to the identity
\begin{equation}
	R(z) = I+\mathcal{O}\left(\frac{1}{n}\right), \qquad n\to\infty, \nonumber
\end{equation}
uniformly for $z \in \mathbb{C} \setminus \Sigma_R$. Here, $\Sigma_R$ is a contour that results from the sequence of transformations outlined before, and consists of the boundaries of the regions in Figure \ref{Fregions}. If we match all powers of $z$ and $n$ in \cref{EjumpR} via $\Delta_k^{\R/\L}(z)$, we obtain higher order terms in the asymptotic expansions, which is exactly the technique outlined in \cref{S:higherorderterms}. Finally, reversing these transformations (since the different RHP are equivalent), one can obtain asymptotic information for $Y(z)$ as $n\to\infty$ in different sectors of the complex plane, and in particular of the $(1,1)$ entry. % <- Comma added

\subsection{The function $R(z)$ in the complex plane} \label{SfctR}

The function $R(z)$ is a $2\times 2$ matrix complex--valued function, analytic (element-wise) in $\mathbb{C}\setminus\Sigma_R$, where $\Sigma_R$ consists of the boundaries of the regions in Figure \ref{Fregions}. Also, $R_k(z) = \mathcal{O}(1/z)$ for $z \to \infty$. Finally, as $n\to\infty$ for polynomial $Q(x)$ with degree $m$ and independent of $n$, %Alfredo/fixed polynomial $Q(x)$ with degree $m$, 
there exist functions $R_k(z)$ %Alfredo/ we zeggen al dat R(z) analytisch moet zijn in C\Sigmaa_R, dus is dat misschien toch niet verwarrend?
such that the function $R(z)$ admits an asymptotic expansion of the form
\begin{equation}\label{asympRn}
		R(z)\sim I + \sum_{k=1}^{\infty} \frac{R_k(z)}{n^{\tfrac{k-1}{m}+1} }, \qquad n \rightarrow \infty.
\end{equation}

We obtain different expressions for $R_k(z)$ depending on the region in which $z$ lies. We will write $R_k^{\R}(z)$ and $R_k^{\L}(z)$ to refer to the coefficients for $z$ near $1$ and $0$ respectively, and $R_k^{\O}(z)$ to indicate the coefficients for $z$ outside these two disks. Our formulas for the asymptotic expansions are written in terms of these functions. One may simply substitute $R(z)=I$ to obtain the leading order behaviour of the expansion. Higher-order expansions are obtained via recursive computation of the $R_k(z)$ in \S \ref{S:higherorderterms}, or alternatively using the explicit expressions listed in \S \ref{APP:explicit}.

\subsection{Auxiliary functions}\label{ss:auxiliary}

We recall some terminology and notation from \cite{Vanlessen}. In the formulation of our results we use the third Pauli matrix $\sigma_3 = \left[ \begin{array}{cc} 1 & 0 \\ 0 & -1 \end{array}\right]$ and define for $a \in \mathbb{C} \setminus \{0\}$:%Alfredo/$a \in \mathbb{C}$:a^{\sigma_3} = \left[ \begin{array}{cc} a & 0 \\ 0 & -a \end{array}\right]. %a^{\sigma_3} = \left[ \begin{array}{cc} a & 0 \\ 0 & \tfrac{1}{a} \end{array}\right].
\[
 a^{\sigma_3} = \left[ \begin{array}{cc} a & 0 \\ 0 & a^{-1} \end{array}\right].
\]
The following values $A_k$ arise in the recursive construction of the MRS numbers:
\[
	A_k =  \int_0^1 \frac{x^{k-1/2}}{\pi\sqrt{1-x}} dx \quad = \prod_{j=1}^k \frac{2j-1}{2j} \quad = \frac{\Gamma(k+1/2)}{\sqrt{\pi}\Gamma(k+1)} \quad = 4^{-k} {2k \choose k}, 
\]
and we will also use the \emph{Pochhammer symbol} or \emph{rising factorial}
\[
	(n)_j =  n(n+1)\cdots(n+j-1). \nonumber
\]

We will define the MRS number $\beta_n$ and the associated quantities $h_n(z), H_n(z)$ and $l_n$ in \cref{Smrs}. The corresponding scaling \eqref{eq:scaling} gives rise to the rescaled field $V_n(z)$:
\[
	V_n(z) = Q(\beta_n z)/n.
\]
We also define the following functions: 
\begin{align}
	\theta(z) &= \begin{cases} 1,& \quad \arg(z-1) > 0, \\ -1, & \quad \arg(z-1)\leq 0, \end{cases} \nonumber \\
	\xi_n(z) &= -i \left(H_n(z)\sqrt{z}\sqrt{1-z}/2 -2\arccos(\sqrt{z})\right), \label{Exin} \\%Alfredo/ Is de "associated quantities ..." die ik er (hierboven) van je terug moest inzetten dan toch niet voldoende? 
	f_n(z) &= n^{2/3}(z-1)\left[\frac{-3\theta(z)\xi_n(z)}{2(z-1)^{3/2}}\right]^{2/3}, \nonumber \\%Alfredo/ Is de "one can avoid specifying select branch cuts by not simplifying the definition of $f_n(z)$" hieronder dan toch niet voldoende?
	\bar{\phi}_n(z) &= \xi_n(z)/2 -\pi i/2. \label{Ephibar}
\end{align}
The function $\theta(z)$ is non-standard in literature on asymptotics, but it is introduced here because it allows the statement of analytic continuations of some functions using standard branch cuts. We assume standard branch cuts of all analytic functions in this paper, such that the formulas are easily implemented. One may call $\xi_n(z)$, $f_n(z)$ and $\bar{\phi}_n(z)$ \emph{phase functions} for the orthonormal polynomials. They specify the oscillatory behaviour of $p_n(x)$ for $z$ respectively away from the endpoints, near $1$ and near $0$. Here, too, one can avoid specifying select branch cuts by not simplifying the definition of $f_n(z)$. The function $\bar{\phi}_n(z)$ corresponds to $\sqrt{\tilde{\varphi}_n(z)}$ in \cite{Vanlessen} and is used for the analytic continuation of the polynomial in the left disk.

The conformal map $\varphi(z)$ from $\mathbb{C} \setminus [0,1]$ onto the exterior of the unit circle is used in the global parametrix $P^{(\infty)}(z)$, which determines the behaviour of $p_n(x)$ away from $x=0$ and $\beta_n$:
\begin{align}
	\varphi(z) &= 2z-1+2\sqrt{z}\sqrt{z-1} \quad = \exp(i \theta(z) \arccos(2z-1)), \label{Evarphi} \\
	P^{(\infty)}(z) &= \frac{2^{-\alpha\sigma_3} }{2 z^{1/4}(z-1)^{1/4}} \begin{pmatrix} \sqrt{\varphi(z) } &  \frac{i}{\sqrt{\varphi(z) } } \\ \frac{-i}{\sqrt{\varphi(z) } } & \sqrt{\varphi(z) } \end{pmatrix}   \left(\frac{z^{\alpha/2}}{(\varphi(z))^{\alpha/2}}\right)^{-\sigma_3 }. \label{EPinf}
\end{align}
Finally, the coefficients $\nu_k$ and $(\alpha,m)$ appear in asymptotics of Airy and modified Bessel functions in the local parametrices:
\begin{align}
	\nu_k &= \left(1-\frac{6k+1}{6k-1}\right) \frac{\Gamma(3k+1/2)}{54^k (k!) \Gamma(k+1/2)} \quad = \frac{-\Gamma(3k-1/2)2^k}{2k 27^k \sqrt{\pi}\Gamma(2k)}, \nonumber \\
	(\alpha,m) &= \begin{cases} 1, & \quad m = 0, \\ 2^{-m} (m!)^{-1} \prod_{n=1}^{m}(4 \alpha^2-(2n-1)^2), & \quad m > 0. \end{cases} \nonumber 
\end{align}

\section{MRS numbers and related functions}
\label{s:MRS}
The Mhaskar-Rakhmanov-Saff numbers $\beta_n$ satisfy~\cite{Vanlessen}
\begin{equation}
	2\pi n = \int_0^{\beta_n} Q'(x)\sqrt{\frac{x}{\beta_n-x}}dx. \label{EbetaInt}%Alfredo/ ook 
\end{equation}
We will explain how to compute these and quantities dependent on them for various types of $Q(x)$: monomials, more general polynomials and more general analytic functions.

\subsection{Monomial $Q(x)$}

If $Q(x)$ is monomial ($Q(x) = q_m x^m + q_0$), then \cite{Vanlessen}
\[
	\beta_n = n^{1/m}\left(m q_m A_m /2\right)^{-1/m}.
\]
We also recall from \cite{Vanlessen} the coefficients $l_n$ and polynomials $H_n$ with a slight adjustment for $l_n$:
\begin{align}
	l_n = & -2/m-4\ln(2) -q_0/n, \nonumber \\
	H_n(z) = & \frac{4}{2m-1} {}_2F_1(1,1-m;3/2-m;z) \quad = \frac{2}{mA_m}\sum_{k=0}^{m-1} A_{m-1-k}z^k. \label{EHn} 
\end{align}
In the classical Laguerre case where $w(x) = x^\alpha e^{-x}$, we have
\begin{align}
	H_n(z) &= 4, \nonumber \\ 
	\beta_n &= 4n. \nonumber
\end{align}
The latter value for $\beta_n$ is well-known and it implies that the largest root of the Laguerre polynomial of degree $n$ grows approximately like $4n$.

\subsection{General polynomial $Q(x)$} \label{Smrs}%Alfredo /\subsection{Asymptotics for general polynomial $Q(x)$} \label{Smrs}
For general polynomial $Q(x)$, $\beta_n$ has an asymptotic expansion with fractional powers,
\begin{equation}
	\beta_n \sim n^{1/m} \sum_{k=0}^\infty \beta^{1,k}n^{-k/m}. \label{EbetaExpa}
\end{equation}
To compute the coefficients $\beta^{1,k}$, we start from the equation at the end of the proof of \cite[Prop 3.4]{Vanlessen}:
\begin{equation}
	\sum_{k=0}^m \tfrac k2 q_k A_k \left. \frac{d^j}{d\epsilon^j}(\beta(\epsilon)^k\epsilon^{m-k})\right|_{\epsilon =0} \quad = \quad 0 \quad = \quad \sum_{k=\max(m-j,1)}^m k q_k A_k \beta^{k,j-m+k}, \label{EstartBeta} 
\end{equation}
where we have defined
\begin{equation}
	\beta(\epsilon)^k \sim \left(\sum_{l=0}^\infty \beta^{1,l} \epsilon^l \right)^k \sim \sum_{l=0}^\infty \beta^{k,l} \epsilon^{l}, \qquad \beta^{k,l} = \sum_{i=0}^l \beta^{k-1,i} \beta^{1,l-i}. \label{Ebetaa}
\end{equation}
One uses the result \cite[(3.8)]{Vanlessen}
\begin{equation}
	\beta^{1,0} = (m q_m A_m/2)^{-1/m}, \label{Ebeta0} 
\end{equation}
and then recursively computes \eqref{Ebetaa} for $l\leq j=0$. Next, \cref{EstartBeta} leads to
\begin{align} 
	\beta^{1,j} = & -\left\{  m q_m A_m\left(\left[\left(\beta^{1,0}\right)^{m-2} \sum_{i=1}^{j-1} \beta^{1,j-i} \beta^{1,i}\right]  + \sum_{i=0}^{m-3} (\beta^{1,0})^i \sum_{k=1}^{j-1} \beta^{1,j-k} \beta^{m-1-i,k} \right) \right. \nonumber \\
	 & \left. + \sum_{k=\max(m-j,1)}^{m-1} k q_k A_k \beta^{k,j-m+k} \right\} \left(m^2 q_m A_m [\beta^{1,0}]^{m-1}\right)^{-1} \nonumber
\end{align}
for $j=1$ and so on. We see that $q_0$ does not influence the MRS number, since it only rescales the weight function. The construction of $\beta_n$ from this section satisfies the condition \cref{EbetaInt} asymptotically up to the correct order and also the following explicit result from \cite[(3.8)]{Vanlessen}:
\begin{equation}
	\beta^{1,1} = \frac{-2(m-1)q_{m-1}}{m(2m-1)q_m}. \label{Ebeta1}
\end{equation}
With these results, we can compute the polynomials $H_n(x)$ and the coefficients $l_n$ as \cite[\S 3]{Vanlessen}: 
\begin{align}
	H_n(z) = & \sum_{k=0}^{m-1} z^k\sum_{j=k+1}^m \frac{q_j}{n} \beta_n^j A_{j-k-1} \quad \sim \sum_{k=0}^{m-1} z^k\sum_{j=k+1}^m q_j A_{j-k-1} n^{j/m-1} \sum_{i=0}^\infty \beta^{j,i} n^{-i/m}, \label{EHn_nonmon} \\ 
	l_n = & -4\log(2) -\sum_{k=0}^m \frac{q_k}{n}\beta_n^k A_k. \nonumber
\end{align}

\subsection{General function $Q(x)$} \label{SmrsNonpoly}%\subsection{General $Q(x)$} \label{SmrsNonpoly}

For the calculation of the functions in case $Q(x)$ is not a polynomial, we introduce a numerical method. We provide an initial guess $\beta_n^{(0)}$ that satisfies $n=Q(\beta_n^{(0)})$ %Alfredo/ Nu met beta_n^0: kan dit nog duidelijker?
to an iterative numerical procedure, so $\beta_n^{(0)} = Q^{-1}(n)$. %procedure which finds a $\beta_n$ that satisfies \cref{EbetaInt}, where the integral is c
The procedure finds a $\beta_n$ that approximately satisfies \cref{EbetaInt}, where the integral is computed by numerical integration. If needed, $Q^{-1}(n)$ and $Q^\prime(x)$ can be approximated numerically as well. The other functions are given by integrals, \cite[(3.11-16-38-40-24)]{Vanlessen} 
\begin{align}
	h_n(z) &=  \frac{1}{2\pi i}\oint_{\Gamma_z} \frac{\sqrt{y}V_n^\prime(y) dy}{\sqrt{y-1}(y-z)}, \nonumber \\
	\xi_n(z) &= \frac{-1}{2} \int_1^z \frac{\sqrt{y-1}}{\sqrt{y}} h_n(y) dy, \nonumber \\
	l_n &=  2\int_0^1 \frac{\log(|1/2-y|)\sqrt{1-y}}{2\pi\sqrt{y}}h_n(y) dy -\frac{Q(\beta_n/2)}{n}. \nonumber
\end{align}
The contour $\Gamma_z$ for $h_n$ should enclose the interval $[0,1]$ and the point $z$. We choose $\Gamma_z$ to be a circle with a center halfway between the interval $[0,1]$ and $z$, while still including $z$ and the interval. The integrals for $\xi_n(z)$ and $l_n$ are also calculated numerically, so the former is computed by a double numerical integral.

\begin{remark} \label{Rgen}
These expressions are also valid for polynomial $Q$. However, following the reasoning in this subsection only leads to a numerical value of $\beta_n$ for a given $n$, as opposed to a full asymptotic expansion of $\beta_n$ in fractional powers of $n$. The same observation holds for the functions defined above. In this case, the powers $n^{-1/m}$ are implicitly present in all quantities that involve $\beta_n$, while the results for polynomial $Q$ are more explicit. We compare both approaches further in \cref{RgenCont,SnumerGen}.
\end{remark}

\subsection{Explicit expressions satisfied by the MRS numbers}

Thus far we have obtained either asymptotic expansions of $\beta_n$ or a numerical estimation. The cases in which explicit expressions can be derived are limited, but in this section we aim to provide some more helpful expressions.

Inspired by \cref{Evarphi}, we invert that conformal map by changing the coordinates $x=(\varphi+1)^2\beta_n/4/\varphi$ in integral \cref{EbetaInt}:%Alfredo/ singulariteit varphi=0 ligt niet op deze contour / Inspired by \cref{Evarphi}, we change the coordinates $x=(\varphi+1)^2\beta_n/4/\varphi$ in integral \cref{EbetaInt}:
\begin{equation}
	\frac{8\pi i n}{\beta_n} = \int_\Upsilon Q^\prime(x)\frac{(\varphi+1)^2}{\varphi^2}  d\varphi.
\end{equation}
The contour $\Upsilon$ is half the unit circle, starting at $\varphi = -1$ through $i$ to $1$. Note that $x$ is real-valued for $\varphi$ on this halfcircle in the upper half of the complex plane. Hence, it is also real-valued when we take the complex conjugate of $\varphi$, corresponding to $\varphi$ on the halfcircle in the lower half of the complex plane. If we assume that $Q^\prime$ is real for real arguments, then the integral on the negative halfcircle, i.e. from $-1$ through $-i$ to $1$, is the complex conjugate of the integral above. Combining both, we find that%Here, we assume that $Q(x)$ is entire to avoid having to substract additional residues, and if $Q(x)$ is real-valued for real arguments, so is the integrand above. Therefore, the integral along the half circle in the lower half of the complex plane, i.e. from $-1$ through $-i$ to $1$, is the complex conjugate of the integral above. Combining both, we find that%Alfredo/ Conjecture weg en hangt idd enkel af van Q' maar ook die integraal over de onderste helft vd cirkel. Veranderd: entire en komma: / The contour $\Upsilon$ is half the unit circle, starting at $\varphi = -1$ through $i$ to $1$. Here, we assume that $Q(x)$ is analytic, and if it is real-valued for real arguments so is the integrand above. Therefore, the integral along the half circle in the lower half of the complex plane, i.e. from $-1$ through $-i$ to $1$. is the complex conjugate of the integral above. Combining both, we find that%entire and comma
\[
 \frac{16\pi i n}{\beta_n} = \int_\Xi Q^\prime(x)\frac{(\varphi+1)^2}{\varphi^2}  d\varphi,
\]%16 iso 8
where $\Xi$ is a circle enclosing the origin $\varphi=0$ in the counterclockwise direction.%Added counterclockwise

The described change of variables maps a point $\varphi$ in the interior of the unit circle to $x \in \mathbb{C} \setminus [0,\beta_n]$. If we $Q$ is not entire, we need to substract additional residues; else, the contour $\Xi$ encloses a single pole at the origin. In the case where $Q(x)$ is a polynomial of degree $m$, we obtain from the residue theorem that $\beta_n$ is the root of a polynomial of degree $m$: %De function $\varphi$ maps a point $z$ in the interior of the unit circle to x \in C \ [0,1]. If we assume that Q is entire, then the contour \Xi encloses a single pole at the origin. In the case where Q(x) is a polynomial of degree m, we obtain from%In the case where $Q(x)$ is a polynomial of degree $m$, we obtain from the residue theorem that $\beta_n$ is the root of a polynomial of degree $m$: 
\begin{align}%	\frac{8n}{\beta_n}  &= \frac{1}{2\pi i} \oint_\Xi Q^\prime(x)\frac{(\varphi+1)^2}{\varphi^2}  d\varphi, \\
	\frac{8n}{\beta_n}  &= \text{Res} \left( \frac{(\varphi+1)^2}{\varphi^2} \sum_{k=1}^m k q_k \left(\frac{(\varphi+1)^2\beta_n}{4\varphi}\right)^{k-1}, \quad \varphi=0 \right), \\
	\frac{8n}{\beta_n}  &= \left[\sum_{k=2}^{m} k q_k \left(\frac{\beta_n}{4}\right)^{k-1} {2k-2 \choose k-2}\right] + 2\left[\sum_{k=1}^m k q_k \left(\frac{\beta_n}{4}\right)^{k-1}  {2k-2 \choose k-1} \right] + \left[\sum_{k=2}^m k q_k \left(\frac{\beta_n}{4}\right)^{k-1}  {2k-2 \choose k} \right], \\
	8n &= 2q_1\beta_n + 4\sum_{k=2}^{m} k {2k \choose k} q_k \left(\frac{\beta_n}{4}\right)^k. \label{EexplBeta}
\end{align}

We can remark that \cref{EbetaExpa} gives the asymptotic expansion of the zero of the $m$-th degree polynomial \cref{EexplBeta} in $\beta_n$ with respect to a factor in its constant coefficient. Exact solutions for $\beta_n$ are only available up to $m=4$. For $m=1$, this boils down to the standard associated Laguerre case $\beta_n=4n/q_1$. For $m=2$ we take the positive solution which also corresponds to \cref{Ebeta0,Ebeta1}:%Moved
\begin{equation}
	\beta_n = \frac{-q_1+\sqrt{q_1^2+24q_2 n}}{3q_2} \sim \sqrt{\frac{8n}{3q_2}} -\frac{q_1}{3q_2} +\sqrt{\frac{q_1^4}{864q_2^3 n}} + \mathcal{O}(n^{-3/2}),
\end{equation}%which corresponds to \cref{Ebeta0,Ebeta1}.

We do find an explicit result for the non-polynomial function $Q(x) = \exp(x)$. In that case, we have 
\begin{align}
	8n &\sim \beta_n \text{Res} \left( \frac{1}{\varphi} \sum_{k=0}^\infty \frac{1}{k!} \left(\frac{\beta_n}{4}\right)^k \left(\varphi^{-1} +2 +\varphi \right)^{k+1}, \varphi=0 \right) \\
	8n &\sim  \beta_n \sum_{k=0}^\infty \frac{1}{k!} \left(\frac{\beta_n}{4}\right)^k {2k+2 \choose k+1} \nonumber \\
	8n &=   2\beta_n \exp\left(\frac{\beta_n}{2}\right)\left[I_0\left(\frac{\beta_n}{2}\right) + I_1\left(\frac{\beta_n}{2}\right) \right] \label{EexponQbeta} \\%Alfredo/ ik denk niet dat het nodig is om te zeggen dat dit gewijzigde Besselfuncties zijn of bedoelt hij het het hem verbaast?
	4n &\sim  e^{\beta_n} \sqrt{\frac{\beta_n}{\pi}} \left[2 - \frac{1}{2\sqrt{\beta_n}} \right] \qquad \Rightarrow \beta_n \sim W(8\pi n^2)/2 \sim \log(n)-\log(\log[8\pi n^2])/2+\log(8\pi)/2, \label{EbetaExp} % asy LambertW toegevoegd, en fouten verbeterd % 4n \sim  e^{\beta_n} \sqrt{\frac{\beta_n}{\pi}} \left[2 + \frac{44}{3\beta_n^2} \right] \qquad \Rightarrow \beta_n \sim W(4\pi n^2) \sim 2\log(n), \label{EbetaExp}
\end{align} 
where $W$ denotes the Lambert-W function, and $I_0$ and $I_1$ are modified Bessel functions. For general $Q(x)$, a similar technique may allow one to find an explicit expression satisfied by $\beta_n$ like \cref{EexponQbeta} without integrals. Solving that expression numerically avoids having to evaluate the integral \cref{EbetaInt}. 
However, it might become quite involved to derive higher order terms as explicitly as in \cref{SprecLnonmon,SprecRnonmon} from the resulting expansion of $\beta_n$ as $n \rightarrow \infty$.%Added %However, it might become quite involved to combine the resulting expansion of $\beta_n$ as $n \rightarrow \infty$ with a derivation of higher order terms comparable to \cref{SprecLnonmon,SprecRnonmon}.%Added
%in each iteration, speeding up computations. A disadvantage is that one might have to recognize the (infinite) series expansion of a possibly involved explicit expression. But in some cases, this could result in an explicit expression for $\beta_n$ in terms of $n$, although the resulting expansion of $\beta_n$ as $n \rightarrow \infty$ will probably involve more general functions than (fractional) powers of $n$. Thus, it will probably not be compatible with the derivation of higher order terms like in \cref{SprecLnonmon,SprecRnonmon}, so that getting similar explicit expressions will be very difficult for general $Q(x)$.

\section{Asymptotics of orthonormal polynomials $p_n(x)$ and related coefficients}
\label{s:asymptotics}

\subsection{Lens \Ri } \label{Sint}

Putting together the consecutive transformations in \cite{Vanlessen}, for $z \in $ \Ri~in \cref{Fregions} and $x$ and $z$ related as in $x=\beta_n z$, we obtain 
\begin{align} %	p_n(x) &= \frac{\beta_n^{n}\gamma_n e^{n (V_n(z)/2+l_n/2)} }{z^{1/4}(1-z)^{1/4} z^{\alpha/2} } \begin{pmatrix} 1 \\ 0 \end{pmatrix}^T R^{\O}(z)  \label{EpiInt} \\ %Alfredo/	p_n(x) &= \frac{\beta_n^{n}\gamma_n e^{n (V_n(z)+l_n)/2} }{z^{1/4}(1-z)^{1/4} z^{\alpha/2} } \begin{pmatrix} 1 \\ 0 \end{pmatrix}^T R^{\O}(z)  \label{EpiInt} \\
	p_n(\beta_n z) &= \frac{\beta_n^{n}\gamma_n e^{n (V_n(z)+l_n)/2} }{z^{1/4}(1-z)^{1/4} z^{\alpha/2} } \begin{pmatrix} 1 \\ 0 \end{pmatrix}^T R^{\O}(z)  \label{EpiInt} \\
	&  \begin{pmatrix} 2^{-\alpha} \cos(\arccos(2z-1)[1/2+\alpha/2] + n\xi_n(z)/i -\pi  /4)  \\  -i 2^{\alpha} \cos(\arccos(2z-1)[\alpha/2-1/2] + n\xi_n(z)/i -\pi /4) \end{pmatrix}.  \nonumber
\end{align}
The asymptotics of $\gamma_n$ are given in \cref{Sgamman}. The full asymptotic expansion of $p_n(x)$ is obtained by substituting the expansion for $R^{\O}(z)$ that we derive later on.

The asymptotic expansions of the orthonormal polynomials all separate two oscillatory terms (phase functions multiplied by $n$) from the non-oscillatory higher order terms. For polynomial $Q(x)$ of degree $m$, the asymptotic expansion truncated after $T$ terms correspond to a relative error of size $\mathcal{O}(n^{-T/m})$. In the special case of a monomial $Q(x) = q_m x^m +q_0$, the relative error improves to $\mathcal{O}(n^{-T})$.%The asymptotic expansions of the orthonormal polynomials all separate two oscillatory terms (phase functions multiplied by $n$) from the non-oscillatory higher order terms. For $\alpha > -1$ and polynomial $Q(x)$ of degree $m$, the asymptotic expansion truncated after $T$ terms correspond to a relative error of size $\mathcal{O}(n^{-T/m})$. In the special case of a monomial $Q(x) = q_m x^m +q_0$, the relative error improves to $\mathcal{O}(n^{-T})$.

\subsection{Outer region \Ro } \label{Sout}

For $z \in$ \Ro, the asymptotic expansion is
\begin{align} %Alfredo/	p_n(x) &= \frac{\beta_n^{n}\gamma_n e^{n (V_n(z)/2+\theta(z)\xi_n(z)+l_n/2)} \exp(i\theta(z)\arccos(2z-1)\alpha/2) }{2 z^{1/4}(z-1)^{1/4}z^{\alpha/2}}\begin{pmatrix}  1 \\ 0 \end{pmatrix}^T R^{\O}(z)   \label{EpiOut} \\
	p_n(\beta_n z) &= \frac{\beta_n^{n}\gamma_n e^{n (V_n(z)/2+\theta(z)\xi_n(z)+l_n/2)} \exp(i\theta(z)\arccos(2z-1)\alpha/2) }{2 z^{1/4}(z-1)^{1/4}z^{\alpha/2}}\begin{pmatrix}  1 \\ 0 \end{pmatrix}^T R^{\O}(z)   \label{EpiOut} \\
	&  \begin{pmatrix} 2^{-\alpha} \exp(i\theta(z)\arccos(2z-1)/2)    \\  -i 2^{\alpha}\exp(-i\theta(z)\arccos(2z-1)/2).  \end{pmatrix} \nonumber
\end{align}
It may appear to be problematic that $\exp(Q(\beta_n z)/2) = \exp(n V_n(z)/2)$ appears in the asymptotic expansions of the polynomials in the complex plane, especially for this region, since this factor grows very quickly. However, one may verify that this exponential behaviour is canceled out with other terms. More specifically, the term $\exp(n\xi_n(z))$ in \eqref{EpiOut} ensures that $p_n(z) = \mathcal{O}(z^n)$, $z \rightarrow \infty$.%$p_n(x) = \mathcal{O}(z^n)$, $z \rightarrow \infty$.

\subsection{Right disk \Rright } \label{Sboun}

The polynomials behave like an Airy function near the right endpoint $z=1$ ($z \in \Rright$). This is typical asymptotic behaviour near a so-called `soft edge', in the language of random matrix theory. Note that the $\theta(z)$ in the following expression removes the branch cut, so that it can be used throughout $\mathbb{C}$, away from $z=0$: 
\begin{align} %	p_n(x) &= \gamma_n \beta_n^{n} \frac{z^{-\alpha/2} \sqrt{\pi}}{z^{1/4} (z-1)^{1/4} } e^{n(V_n(z)/2 + l_n/2)} \begin{pmatrix} 1 \\ 0 \end{pmatrix}^T R^{\R}(z)  \label{Epiboun} \\ % Alfredo/ 	p_n(x) &= \gamma_n \beta_n^{n} \frac{z^{-\alpha/2} \sqrt{\pi}}{z^{1/4} (z-1)^{1/4} } e^{n(V_n(z) + l_n)/2} \begin{pmatrix} 1 \\ 0 \end{pmatrix}^T R^{\R}(z)  \label{Epiboun} \\
	p_n(\beta_n z) &= \gamma_n \beta_n^{n} \frac{z^{-\alpha/2} \sqrt{\pi}}{z^{1/4} (z-1)^{1/4} } e^{n(V_n(z) + l_n)/2} \begin{pmatrix} 1 \\ 0 \end{pmatrix}^T R^{\R}(z)  \label{Epiboun} \\
	& \begin{pmatrix} 2^{-\alpha}\left\{\cos\left[\frac{(\alpha+1)\arccos(2z-1)}{2} \right] Ai(f_n(z)) f_n(z)^{1/4}  -i\sin\left[\frac{(\alpha+1)\arccos(2z-1)}{2} \right] Ai'(f_n(z))f_n(z)^{-1/4} \theta(z)\right\} \\ 2^\alpha \left\{-i\cos\left[\frac{(\alpha-1)\arccos(2z-1)}{2} \right] Ai(f_n(z)) f_n(z)^{1/4} -\sin\left[\frac{(\alpha-1)\arccos(2z-1)}{2}\right] Ai'(f_n(z)) f_n(z)^{-1/4} \theta(z) \right\}\end{pmatrix} \nonumber
\end{align}

We would like to note a possible issue when computing the zeros of this expression, as one would do for Gaussian quadrature. The largest root for the standard associated Laguerre polynomials asymptotically behaves as $4n+2\alpha+2 + 2^{2/3}a_1(4n+2\alpha+2)^{1/3}$, with $a_1$ the (negative) zero of the Airy function closest to zero, see \cite[(18.16.14)]{DLMF,Olver:2010:NHMF} and \cite[(6.32.4)]{Szego}%Alfredo/ Ref Szego toegevoegd %zero of the Airy function closest to zero \cite[(18.16.14)]{DLMF,Olver:2010:NHMF}%\cite[18.16.14]{DLMF,Olver:2010:NHMF}
. For a fixed but relatively high $\alpha$, this point may lie outside the support of the equilibrium measure $(0,4n)$ \cite[Rem. 3.8]{Vanlessen}. There will always be a larger $n$ for which the point lies inside. Still, for large $\alpha$ one may want to pursue a different kind of asymptotic expansion, for example using asymptotics with a varying weight $x^{\alpha(n)}\exp(-Q(x))$, as studied in for example \cite{varying}, and apply it to the fixed $\alpha$.%Alfredo/ Gawronski in varying genoeg? En ik denk dat Edmond Laguerre in zijn tijd nog geen DOI kon aanvragen :-)

\subsection{Left disk \Rleft } \label{Slboun}

The polynomials behave like a Bessel function of order $\alpha$ near the left endpoint $z=0$. For $z \in $ \Rleft, we obtain
\begin{align} %	& p_n(x) = \gamma_n \beta_n^{n} \frac{(-1)^n \left(in\bar{\phi}_n(z)\pi\right)^{1/2} }{z^{1/4} (1-z)^{1/4}} z^{-\alpha/2} e^{n(V_n(z)/2 + l_n/2)} \begin{pmatrix} 1 \\ 0 \end{pmatrix}^T R^{\L}(z)   \label{Epilboun} \\ %Alfredo/	& p_n(x) = \gamma_n \beta_n^{n} \frac{(-1)^n \left(in\bar{\phi}_n(z)\pi\right)^{1/2} }{z^{1/4} (1-z)^{1/4}} z^{-\alpha/2} e^{n(V_n(z) + l_n)/2} \begin{pmatrix} 1 \\ 0 \end{pmatrix}^T R^{\L}(z)   \label{Epilboun} \\
	& p_n(\beta_n z) = \gamma_n \beta_n^{n} \frac{(-1)^n \left(in\bar{\phi}_n(z)\pi\right)^{1/2} }{z^{1/4} (1-z)^{1/4}} z^{-\alpha/2} e^{n(V_n(z) + l_n)/2} \begin{pmatrix} 1 \\ 0 \end{pmatrix}^T R^{\L}(z)   \label{Epilboun} \\
	& \begin{pmatrix} 2^{-\alpha} \left\{ \sin\left[\frac{(\alpha+1)\arccos(2z-1)}{2} -\frac{\pi\alpha}{2}\right] J_\alpha\left(2in \bar{\phi}_n(z)\right) + \cos\left[\frac{(\alpha+1)\arccos(2z-1)}{2} -\frac{\pi\alpha}{2}\right] J_\alpha'\left(2in \bar{\phi}_n(z)\right) \right\} \\ -i 2^\alpha \left\{ \sin\left[ \frac{(\alpha-1)\arccos(2z-1)}{2} -\frac{\pi\alpha}{2}\right] J_\alpha\left(2in \bar{\phi}_n(z)\right) +  \cos\left[\frac{(\alpha-1)\arccos(2z-1)}{2} -\frac{\pi\alpha}{2}\right] J_\alpha'\left(2in \bar{\phi}_n(z)\right) \right\} \end{pmatrix}. \nonumber
\end{align}

It is not immediately obvious that the expansions \eqref{Epiboun} and \eqref{Epilboun} are analytic in the points $z=1$ and $z=0$ respectively. This will follow from the expression \cref{ERrlseries} for $R^{\L/\R}(z)$ and by also making a series expansion of the other terms at those points. For numerical purposes, it may be better to use those series expansions when evaluating close to (or at) $z=0$ and $z=1$.

\subsection{Asymptotics of leading order coefficients} \label{Sgamman}

The leading order coefficient of the orthonormal polynomials is $\gamma_n$, i.e. we have
\[
p_n(x) = \gamma_n\pi_n(x)
\]
where $\pi_n(x)$ is the monic orthogonal polynomial of degree $n$. For a monomial or more general function $Q(x)$, the asymptotic expansion of $\gamma_n$ is 
\begin{align} 
	\gamma_n & \sim \frac{\beta_n^{-n-\alpha/2-1/2}\exp(-n l_n/2) \sqrt{\frac{2}{\pi}} 2^{\alpha} } {\sqrt{1-4i4^{\alpha} \sum_{k=1}^\infty \frac{(U_{k,1}^{\R}+U_{k,1}^{\L} )|_{1,2} } {n^k} } }. \label{Egamma}
\end{align}
The quantities $U_{k,1}^{\R/\L}$ are defined and extensively described in \S\ref{S:higherorderterms}. They are the constant $2\times 2$ matrices that multiply $z^{-1}n^{-k}$ and $(z - 1)^{-1} n^{-k}$ in the expansion for $R(z)$, of which we use the lower left elements here. Explicit expressions for these matrices up to $k=3$ are given in Appendix \ref{APP:explicit}. The constant coefficient $q_0$ only changes the scaling of the weight function and does not influence $\beta_n$ nor the matrices. However, it does influence $\gamma_n$ through the coefficient $l_n$, giving $\gamma_n \sim \exp(q_0/2)$.

For general polynomial $Q(x)$, the power of $n$ changes from $k$ to $(k-1)/m+1$: 
\begin{align} 
	\gamma_n & \sim \frac{\beta_n^{-n-\alpha/2-1/2}\exp(-n l_n/2) \sqrt{\frac{2}{\pi}} 2^{\alpha} } {\sqrt{1-4i4^{\alpha} \sum_{k=1}^\infty \frac{(U_{k,1}^{\R}+U_{k,1}^{\L} )|_{1,2} } {n^{(k-1)/m+1} } } } . \nonumber
\end{align}
This reflects the more accurate asymptotic information about $\beta_n$ available for polynomial $Q$. It is understood here that one substitutes the asymptotic expansion of $\beta_n$ in this formula. We retain this formulation here to show the analogy with \cref{Egamma}.

\subsection{Asymptotics of recurrence coefficients}
In the three term recurrence relation \eqref{recur}, the recurrence coefficients have the following large $n$ asymptotic expansion 
\begin{align}
	& a_n \sim \beta_n \left[\frac{-\alpha}{4} + \sum_{k=1}^\infty \frac{(U_{k,1}^{\R}+U_{k,1}^{\L} )|_{1,1} }{n^k}  + \right. \nonumber \\
	 & \left. \frac{ \frac{4^{-\alpha}i(\alpha+2)}{16} + \left(\sum_{k=3}^\infty\frac{U_{k,2}^{\L}|_{1,2} }{n^k} \right) + \left(\sum_{k=1}^\infty\frac{(U_{k,2}^{\R}+U_{k,1}^{\R} )|_{1,2} }{n^k} \right) + \left(\sum_{k=1}^\infty \frac{(U_{k,1}^{\R}+U_{k,1}^{\L} )|_{1,1} 4^{-\alpha-1}i + (U_{k,1}^{\R}+U_{k,1}^{\L} )|_{1,2}\alpha/4 }{n^k}  \right) } {4^{-\alpha-1}i + \left(\sum_{k=1}^\infty \frac{(U_{k,1}^{\R}+U_{k,1}^{\L} )|_{1,2} }{n^k}\right)   } \right] \nonumber
\end{align}
and
\begin{align}\nonumber 
	b_{n-1} \sim \frac{\beta_n}{4} & \left[1+4i\left(\sum_{k=1}^\infty \frac{(U_{k,1}^{\R}+U_{k,1}^{\L} )|_{2,1}4^{-\alpha} -4^\alpha (U_{k,1}^{\R}+U_{k,1}^{\L} )|_{1,2}    }{n^k}\right)  +  \right. \nonumber \\
	 & \left. 16\left(\sum_{k=1}^\infty \frac{(U_{k,1}^{\R}+U_{k,1}^{\L} )|_{2,1} }{n^k}\right) \left(\sum_{k=1}^\infty \frac{(U_{k,1}^{\R}+U_{k,1}^{\L} )|_{1,2} }{n^k}\right) \right]^{1/2}.
\end{align}
The quantities $U_{k,1}^{\R/\L}$ in these expressions are the same as those appearing in \eqref{Egamma} above. For general polynomial $Q(x)$, the powers of $n$ again change from $k$ to $(k-1)/m+1$ and the expression is otherwise unchanged.

\section{Computation of higher-order terms}\label{S:higherorderterms}

The literature on Riemann-Hilbert problems suggests a way to compute the asymptotic expansion of $R(z)$. In principle, it is clear how expressions can be obtained but this involves many algebraic manipulations, summations and recursion. An important contribution of \cite{jacobi} was to identify a set of simplifications that significantly improve the efficiency of numerical evaluation of the expressions. Similar simplifications can be performed in the current setting of Laguerre-type polynomials, though the expressions are of course very different.

\subsection{Jumps of $R(z)$}\label{ss:localjump}

The main idea to obtain higher-order terms in the asymptotic expansion for $\pi_n(z)$ is to compute the higher-order terms $R_k(z)$ in \eqref{asympRn}. To that end, we recall that $R$ satisfies a Riemann-Hilbert problem with jumps across the contours shown in \cref{Fregions}. We proceed as in \cite{jacobi} by writing the jump matrix for $R(z)$ as a perturbation of the identity matrix, $I+\Delta(z)$. Starting from \cite[(3.108)]{Vanlessen} we have, for $z$ on the boundary $\Sigma_R$ of one of the disks as shown in \cref{Sasy},
\begin{equation}\label{Ejump2}
	R^{\O}(z) = R^{\R/\L}(z)(I+\Delta^{\R/\L}(z)), \qquad z \in \Sigma_R.
\end{equation}
We then consider a full asymptotic expansion in powers of $1/n$ for $\Delta(z)$:
\begin{equation}\nonumber
   \Delta(z) \sim \sum_{k=1}^{\infty} \frac{\Delta_k(z)}{n^k}, \qquad n\to\infty.
\end{equation}
On the boundary of the disks, the terms $\Delta_k(z)$ can be written explicitly as $\Delta_k^{\R/\L}(z)$ \cite[(3.76), (3.98)]{Vanlessen}: 
\begin{align}
	\Delta_k^{\R}(z) & = \frac{P^{(\infty)}(z)z^{\alpha/2\sigma_3}}{2\left(- \xi_n(z) \right)^k} \begin{pmatrix} (-1)^k\nu_k & -6k i\nu_k \\ 6k i(-1)^k\nu_k & \nu_k \end{pmatrix} z^{-\alpha/2\sigma_3} P^{(\infty)}(z)^{-1}, \quad 0 < |z-1| < \delta_2 \\
	\Delta_k^{\L}(z)& = \frac{(\alpha,k-1)}{\left(4\bar{\phi}_n(z) \right)^k}P^{(\infty)}(z) (-z)^{\alpha/2\sigma_3}    \begin{pmatrix} \tfrac{(-1)^k}{k}(\alpha^2+\tfrac{k}{2} -\tfrac{1}{4}) & \left(k-\tfrac{1}{2}\right)i \\ (-1)^{k+1}\left(k-\tfrac{1}{2}\right)i & \tfrac{1}{k}(\alpha^2 +\tfrac{k}{2} -\tfrac{1}{4}) \end{pmatrix} (-z)^{-\alpha/2\sigma_3} P^{(\infty)}(z)^{-1}, 
\end{align} 
with $0 < |z| < \delta_3$ for some sufficiently small $\delta_2$ and $\delta_3 > 0$. However, for a general polynomial or function $Q(x)$, the $\Delta_k^{\R/\L}(z)$ are also dependent on $n$. %Alfredo/ $Q(x)$, they are dependent on $n$.
We will extract the $n$-dependence explicitly for general polynomial $Q(x)$, and use contour integrals for each required $n$ otherwise. $\Delta_k^{\L}(z)$ has poles of order at most $\lceil k/2 \rceil$ at $z=0$ \cite[Rem. 3.29]{Vanlessen} as in the Jacobi case \cite{jacobi}, but $\Delta_k^{\R}(z)$ has poles of order at most $\lceil 3k/2 \rceil$ at $z=1$ \cite[Rem. 3.22]{Vanlessen}. %Alfredo/ Inderdaad zeer belangrijk en ontbrak
The $\Delta_k(z)$ are identically $0$ on the other boundaries of the regions in \cref{Fregions}. 

\begin{remark}\label{Rspec}
If $\alpha^2 = 1/4$ as in the Hermite case (see \cref{Sherm}), then $\Delta^{\L}_k(z)$ and $s^{\L}_k(z)$ are zero matrices for $k > 1$ and $\Delta^{\L}_1(z)$ and $s^{\L}_1(z)$ have a Taylor series starting with $\mathcal{O}(1)$ near $z=0$. So, all $U_{k,m}^{\L}$ are zero matrices and can be left out of the calculation of higher order terms, which is still needed as the $U_{k,m}^{\R}$ are not zero. 
\end{remark}

\subsection{Recursive computation of $R_k(z)$ for monomial $Q(x)$} \label{ScompU}

In this case, there are no fractional powers of $n$ involved, and we can renumber \cref{asympRn} to simplify the formulas:
\begin{equation}
  \label{asympRnmon}
  R(z)\sim I + \sum_{k=1}^{\infty} \frac{R_k(z)}{n^{k} }, \qquad n \rightarrow \infty. 
\end{equation}
By expanding the jump relation \eqref{Ejump2} and collecting the terms with equal order in $n$, we obtain a link between the terms $R_k(z)$ in the expansion \eqref{asympRn} and the $\Delta_k$. For $z$ on the boundary of the disks in \cref{Fregions}, we have 
\begin{equation}\label{RHPforRk}
	R_k^{\O}(z)=R_k^{\R/\L}(z)+\sum_{j=1}^k R^{\R/\L}_{k-j}(z)\Delta^{\R/\L}_j(z)
\end{equation}
with $R_0^{\R/\L}(z)=I$. One can solve the additive Riemann-Hilbert problem as follows:
\begin{itemize}
 \item Expand the sum in \eqref{RHPforRk} in a Laurent series around $z=0$ and $1$.
 \item Define $R_k^{\O}(z)$ as the sum of all the terms containing strictly negative powers of $z$ and $(z- 1)$. Since $R_k(z) = \mathcal{O}(1/z)$ as $z \rightarrow \infty$, positive powers do not contribute to $R_k^{\O}(z)$.
 \item Define $R^{\R}_k(z)$ as the remainder after subtracting those poles.
\end{itemize}
This construction ensures that $R_k^{\O}$ is analytic outside the disk, $R^{\R}_k$ is analytic inside and \eqref{RHPforRk} holds, as required. According to \cite[Rem. 3.22 \& 3.29]{Vanlessen}, we may write
\begin{equation}\label{Vkm}
	\Delta_k^{\R/\L}(z) \sim \sum_{i=-\lceil 3k/2 \rceil }^\infty V_{k,i}^{\R/\L} (z -1/2 \mp 1/2)^i, 
\end{equation}
with $V_{k,p}^{\L} \equiv 0$ for all $p < -\lceil k/2 \rceil$. Note that $V_{k,i}^{\R/\L}$ are the Laurent coefficients of $\Delta_k^{\R/\L}$ around $z=1$ and $z=0$ respectively. %Alfredo/ Empty note?
With $U_{k,p}^{\L} \equiv 0$ for all $p < -\lceil k/2 \rceil$, this yields
\begin{equation}\label{ERpl}
	R_{k}^{\O}(z) =\sum_{p=1}^{\lceil 3k/2 \rceil } \left(\frac{U_{k,p}^{\R} }{(z-1)^p} + \frac{U_{k,p}^{\L} }{z^p} \right).
\end{equation}
At the same time, since $R^{\R/\L}_k(z)$ are analytic in z=$1$ (respectively $0$), 
\begin{equation}
	R^{\R}_k(z) \sim \sum_{n=0}^\infty Q_{k,n}^{\R} (z-1)^n, \quad 	R^{\L}_k(z) \sim \sum_{n=0}^\infty Q_{k,n}^{\L} z^n, \label{ERrlseries} 
\end{equation}
with some coefficients $Q^{\R/\L}_{k,n}$ that can be determined as well, for example via symbolic differentiation. It follows from the additive jump relation \eqref{RHPforRk}, that
\begin{align}
	U_{k,p}^{\R/\L} = & \hspace{1.5mm} V_{k,-p}^{\R/\L} + \sum_{j=1}^{k-1}\sum_{l=0}^{\lceil 3j/2 \rceil -p} Q_{k-j,l}^{\R/\L} V_{j,-p-l}^{\R/\L}, \label{EUpole}  \\
	Q_{k,n}^{\R/\L} = & \left(\sum_{i=1}^{\lceil 3k/2 \rceil} {-i \choose n} (\pm 1)^{i+n} U_{k,i}^{\L/\R} \right)  \\%\left(\sum_{i=1}^{\lceil 3k/2 \rceil} {-i \choose n} (\pm 1)^{-i-n} U_{k,i}^{\L/\R} \right)  \\ %:changed 22dec
	& -V_{k,n}^{\R/\L} -\sum_{j=1}^{k-1}\sum_{l=0}^{\lceil 3j/2 \rceil +n} Q_{k-j,l}^{\R/\L} V_{j,n-l}^{\R/\L}. \nonumber
\end{align}%Here, the $(- 1)^{-i-n} U_{k,i}^{\R}$% changed 22dec:
Here, the $(- 1)^{i+n} U_{k,i}^{\R}$ corresponds to the $Q_{k,n}^{\L}$. In \cref{Ssimpl}, we will explore an alternative way to compute the matrices $U_{k,m}^{\R/\L}$.

\subsection{Recursive computation of $R_k(z)$ for general polynomial $Q(x)$}

For general polynomial $Q(x)$, the $\Delta_k(z)$ are also dependent on $n$, so we need fractional powers of $n$. We introduce Laurent coefficients with an extra index, indicating the power of $n^{-1/m}$,
\begin{equation}
\label{Vkil}
  \Delta^{\R/\L}_k(z) \sim \sum_{l=0}^\infty \left[\sum_{i=-\lceil 3k/2 \rceil}^\infty V_{k,i,l}^{\R/\L}(z -1/2 \mp 1/2)^{i} \right] n^{-l/m}. 
\end{equation}

In this case, we do have a general expansion for $R$ in terms of fractional powers, given earlier by \eqref{asympRn}. We arrive at Taylor series and Laurent expansions of the form:
\begin{align}
	R^{\R/\L}_k(z) &\sim  \sum_{n=0}^\infty Q_{k,n}^{\R/\L} (z -1/2 \mp 1/2)^n, \label{ERQnonmon} \\
	R^{\O}_k(z) &\sim \sum_{p=1}^{\lceil 3/2 \lceil k/m\rceil \rceil } \left(\frac{U_{k,p}^{\R} }{(z-1)^p} + \frac{U_{k,p}^{\L} }{z^p} \right).
\end{align}%Here, $p$ corresponds to the order of the pole and must be $\leq \lceil k/m \rceil$. %Corrected 22dec:
Here, $p$ corresponds to the order of the pole and must be $\leq \lceil 3/2 \lceil k/m\rceil \rceil$. This is because that is the highest order of the pole of the $\Delta^{\L/\R}_q(z)$ matrices which appear in the expansion up to $\mathcal{O}\left(n^{\tfrac{k-1}{m}+1}\right)$ of the jump relation
\begin{align}
	I + & \sum_{k=1}^{\infty} \frac{R_k^{\O}(z)}{n^{\tfrac{k-1}{m}+1} } = \left(I + \sum_{k=1}^{\infty} \frac{R_k^{\R/\L}(z)}{n^{\tfrac{k-1}{m}+1} }\right) \left(I+\sum_{q=1}^{\infty} \frac{\Delta_q(z)}{n^q}\right).
\end{align}

Expanding near $z=0$ or $1$ and collecting terms with the same (fractional) power of $n$ in the jump relation \cref{Ejump2}, we obtain, after some more algebraic manipulations,
\begin{align} 
	U_{k,p}^{\R/\L} = & \left[\sum_{q=\lceil (2p-1)/(2\pm 1) \rceil}^{(k-1)/m +1} V_{q,-p,k-1+m-qm}^{\R/\L}  \right] + \left[ \sum_{q=1}^{(k-1)/m} \sum_{l=0}^{k-1-mq} \sum_{i=-\lceil (2\pm 1)q/2 \rceil}^{-p} Q_{k-l-mq,-i-p}^{\R/\L} V_{q,i,l}^{\R/\L} \right], \\
	Q_{j,n}^{\R/\L} = & \left[\sum_{p=1}^{\lceil (2 \mp 1)/2 \lceil j/m\rceil \rceil } U_{j,p}^{\L/\R} {-p \choose n} (\pm 1)^{n-p}  \right] - \left[ \sum_{q=1}^{(j-1)/m+1} V_{q,n,j-1+m+qm}^{\R/\L} \right] \\
	& -\left[ \sum_{q=1}^{(j-1)/m} \sum_{l=0}^{j-1-qm} \sum_{i=-\lceil (2 \pm 1) q/2 \rceil}^n Q_{j-l-qm,n-i}^{\R/\L} V_{q,i,l}^{\R/\L} \right].
\end{align}

\subsection{Simplifications} \label{Ssimpl}
We start by writing the jump relation \eqref{RHPforRk} using the coefficients $R^{\O}_{k-m}(z)$ instead of $R^{\R/\L}_{k-m}(z)$. 
\begin{proposition} The jump relation \eqref{RHPforRk} can be written as follows:
\begin{equation}\label{E:simplifiedjump}
	R^{\R/\L}_k(z) = R_k^{\O}(z) - \sum_{l=1}^k R_{k-l}^{\O}(z)s^{\R/\L}_l(z)
\end{equation}
with $R_{0}^{\R/\L}(z) = I$ and with
\begin{equation}\label{EcheckS}
   s^{\R/\L}_l(z) = \Delta^{\R/\L}_l(z) - \sum_{j=1}^{l-1} s^{\R/\L}_j(z)\Delta^{\R/\L}_{l-j}(z).
\end{equation} 
\end{proposition}
\begin{proof} This can be proven by induction as in \cite[\S 4.1]{jacobi}. \end{proof}
This formulation has two advantages:
\begin{itemize}
 \item The jump term in \eqref{E:simplifiedjump} is written in terms of $R_{k-l}^{\O}$ rather than $R^{\R}_{k-l}$, and the former has a simple and non-recursive expression \eqref{ERpl}.
 \item The definition of the coefficients $s_l^{\R/\L}$ can be greatly simplified to a non-recursive expression too, involving just the $\Delta_k^{\R/\L}$'s.
\end{itemize}

More precisely, we have the following result:

\begin{proposition} \label{Tsimpl} The terms $s^{\R/\L}_l(z)$ defined by \eqref{EcheckS} satisfy 
\begin{equation}
   s^{\L/\R}_l(z) = \Delta^{\L/\R}_l(z) \nonumber
\end{equation}
for odd $m$ and 
\begin{align}
   s^{\L}_l(z) = & \hspace{1.5mm} \Delta^{\L}_l(z) -\frac{4\alpha^2+2l-1}{(4\bar{\phi}_n(z))^l}\frac{(\alpha,l-1)}{2l} I, \nonumber \\
   s^{\R}_l(z) = & \hspace{1.5mm} \Delta^{\R}_l(z) -\frac{\nu_l}{(-\xi_n(z))^l}I, \nonumber 
\end{align}
for even $m$, where $\xi_n$ and $\bar{\phi}_n$ are as defined in \cref{ss:auxiliary}.%Alfredo/ Evt verwijzen naar sectie Aux fcts maar dan ook bij def \Delta_k's
\end{proposition}
\begin{proof}
This can be proven again by mathematical induction, completely analogous to \cite[\S B]{jacobi} for the left case. We should note that gcd$(q_n,r_{n+j})$ is not necessarily one, which is needed in \cite[under (10d)]{Gosper}, but that the suggested change in variables can eliminate the common factor. For the right case, the proof is also analogous, but with $\lambda_{j,l-j} = a_j = (-1)^j[36(l-j)j -1]\nu_{l-j}\nu_j/2, p_j = 1-36j(l-j), q_j = -(6j-5)(6j-7) (l-j+1)$ and $f_j=1/l$.
\end{proof}

In the next equations, the coefficients $W_{k,i}^{\R/\L}$ are the Laurent coefficients of $s_k^{\R/\L}$ for monomial $Q$,
\begin{equation}
 \label{Wkm}
 s_k^{\R/\L}(z) \sim \sum_{i=-\lceil 3k/2 \rceil}^\infty W_{k,i}^{\R/\L} (z -1/2 \mp 1/2)^i.
\end{equation}
They are used to compute $U_{k,p}^{\R/\L}$ directly, based on \eqref{E:simplifiedjump}. This has the advantage of needing less memory, as the $Q_{k,n}^{\R/\L}$ are not needed any more. Still for monomial $Q(x)$, that leads us to 
\begin{align} 
	U_{k,p}^{\R/\L} &= \hspace{1.5mm} W_{k,-p}^{\R/\L} + 
	\sum_{j=1}^{k-1}\sum_{l=\max(p-\lceil  (2 \pm 1)j/2 \rceil,1) }^{\lceil (2 \pm 1) (k-j)/2 \rceil} U_{k-j,l}^{\R/\L} W_{j,l-p}^{\R/\L}   \label{EUW} \\
	 & \hspace{1.5mm} +\sum_{j=1}^{k-1}\sum_{n=0}^{\lceil (2 \pm 1) j/2 \rceil -p}  \left(\sum_{i=1}^{\lceil (2 \mp 1) (k-j)/2 \rceil} (\pm 1)^{i+n} { -i \choose n} U_{k-j,i}^{\L/\R} \right) W_{j,-n-p}^{\R/\L}. \nonumber 
\end{align}

It might be necessary to approximate the orthonormal polynomials near $z=0$ and $1$, for which it is inaccurate and computationally expensive to use \cref{E:simplifiedjump} and certainly the recursive application of \cref{RHPforRk}. So optionally, one can still compute the series expansion of $R^{\R/\L}_k(z)$ afterwards, using 
\begin{align} 
	Q_{k,n}^{\R/\L} = & \hspace{1.5mm} \left(\sum_{i=1}^{\lceil (2 \mp 1)k/2 \rceil} {-i \choose n} (\pm 1)^{-i-n} U_{k,i}^{\L/\R} \right)-W_{k,n}^{\R/\L} -\sum_{j=1}^{k-1}\sum_{i=n+1}^{n+\lceil (2 \pm 1)(k-j)/2 \rceil} U_{k-j,i-n}^{\R/\L} W_{j,i}^{\R/\L}  \\
	& -\sum_{j=1}^{k-1}\sum_{i=-\lceil (2\pm1) j/2 \rceil}^{n} \sum_{l=1}^{\lceil (2 \mp 1)(k-j)/2 \rceil} { -l \choose n-i} (\pm 1)^{i-n+l} U_{k-j,l}^{\L/\R} W_{j,i}^{\R/\L}. 
\end{align}

For general polynomial $Q(x)$, we have the more general expansion
\begin{equation}\label{Wkmnonmon} 
  s_k^{\R/\L}(z) \sim \sum_{l=0}^\infty \sum_{i=-\lceil 3k/2 \rceil}^\infty W_{k,i,l}^{\R/\L} (z -1/2 \mp 1/2)^i n^{-l/m},
\end{equation}
which leads to
\begin{align} %Alfredo q (eigenlijk q_n en r_n al verder gebruikt maar wss geen verwarring mogelijk)/ 	U_{k,y}^{\R}  = & \left[\sum_{j=\lfloor 2y/3\rfloor}^{\lfloor (k-1)/m \rfloor} W_{j+1,-y,k-1-jm}^{\R}\right] + \sum_{l=0}^{k-1-m} \sum_{j=0}^{\lfloor (k-1-l)/m-1 \rfloor}  \sum_{i=\max(-\lceil 3j/2 \rceil,1-y)}^{\lceil 3l/2 \rceil-y} U_{l,y+i}^{\R} W_{j+1,i,k-1-l-jm-m}^{\R} \\  & + \sum_{l=0}^{k-1-m} \sum_{p=1}^{\lceil l/2 \rceil } U_{l,p}^{\L} \sum_{j=0}^{\lfloor (k-1-l)/m-1 \rfloor}  \sum_{i=-\lceil 3j/2 \rceil}^{-y} W_{j+1,i,k-1-l-jm-m}^{\R}  {-p \choose -y-i}, \\  U_{k,y}^{\L} = & \left[\sum_{j=2y}^{\lfloor (k-1)/m \rfloor}  W_{j+1,-y,k-1-jm}^{\L} \right] + \sum_{l=0}^{k-1-m} \sum_{j=0}^{\lfloor (k-1-l)/m-1 \rfloor} \sum_{i=\max(-\lceil j/2 \rceil,1-y)}^{\lceil l/2 \rceil -y } U_{l,y+i}^{\L} W_{j+1,i,k-1-l-jm-m}^{\L} \\  & \quad + \sum_{l=0}^{k-1-m} \sum_{p=1}^{\lceil 3l/2 \rceil} U_{l,p}^{\R} \sum_{j=0}^{\lfloor (k-1-l)/m-1 \rfloor} \sum_{i=-\lceil j/2 \rceil}^{-y} W_{j+1,i,k-1-l-jm-m}^{\L} (-1)^{y+i+p} {-p \choose -y-i}.
	U_{k,q}^{\R}  = & \left[\sum_{j=\lfloor 2q/3\rfloor}^{\lfloor (k-1)/m \rfloor} W_{j+1,-q,k-1-jm}^{\R}\right] + \sum_{l=0}^{k-1-m} \sum_{j=0}^{\lfloor (k-1-l)/m-1 \rfloor}  \sum_{i=\max(-\lceil 3j/2 \rceil,1-q)}^{\lceil 3l/2 \rceil-q} U_{l,q+i}^{\R} W_{j+1,i,k-1-l-jm-m}^{\R} \\
	& + \sum_{l=0}^{k-1-m} \sum_{p=1}^{\lceil l/2 \rceil } U_{l,p}^{\L} \sum_{j=0}^{\lfloor (k-1-l)/m-1 \rfloor}  \sum_{i=-\lceil 3j/2 \rceil}^{-q} W_{j+1,i,k-1-l-jm-m}^{\R}  {-p \choose -q-i}, \\
	U_{k,q}^{\L} = & \left[\sum_{j=2q}^{\lfloor (k-1)/m \rfloor}  W_{j+1,-q,k-1-jm}^{\L} \right] + \sum_{l=0}^{k-1-m} \sum_{j=0}^{\lfloor (k-1-l)/m-1 \rfloor} \sum_{i=\max(-\lceil j/2 \rceil,1-q)}^{\lceil l/2 \rceil -q } U_{l,q+i}^{\L} W_{j+1,i,k-1-l-jm-m}^{\L} \\
	& \quad + \sum_{l=0}^{k-1-m} \sum_{p=1}^{\lceil 3l/2 \rceil} U_{l,p}^{\R} \sum_{j=0}^{\lfloor (k-1-l)/m-1 \rfloor} \sum_{i=-\lceil j/2 \rceil}^{-q} W_{j+1,i,k-1-l-jm-m}^{\L} (-1)^{q+i+p} {-p \choose -q-i}.
\end{align}

\section{Explicit series expansions for $s_k^{\L/\R}(z)$ and $\Delta_k^{\L/\R}(z)$} \label{SexplL}

In this section we derive fully explicit expressions for the coefficients $W_{k,i}^{\R/\L}$, defined by \eqref{Wkm} or \eqref{Wkmnonmon} for a monomial, general polynomial and general function $Q$ respectively. %for monomial and general polynomial $Q$ respectively.
These expressions are amenable to implementation without further symbolic manipulations. The process and terminology of symbols mimicks that used in \cite{jacobi} for Jacobi polynomials. In this section we aim to be coincise yet complete (thus needing Russian characters): we expand $s_k^{\L/\R}(z)$ and $\Delta_k^{\L/\R}(z)$ in power series where the coefficients are computed using convolutions.%Alfredo/ Ik denk dat men met deze inleiding ook al doorheeft dat dit technisch maar noodzakelijk is, dus hoeft niet in appendix (ook niet in Jacobi) %In this section we aim to be complete (thus needing Russian characters) yet concise.

\subsection{Left disk with monomial $Q(x)$} \label{SexplLeft}

First we consider $s_k^{\L}(z)$, where we know that $W_{k,i}^{\L} \equiv 0$ $\forall i < -\lceil k/2 \rceil$. We have
\begin{equation} \label{EDeltak_Tk}
 \Delta_k^{\L}(z) = \frac{(\alpha,k-1)}{4^k \bar{\phi}_n(z)^k} 2^{-\alpha\sigma_3} G_k(z) 2^{\alpha\sigma_3},
\end{equation}
where $G_k$ is defined for odd and even $k$ as
\begin{align}
	G_k^{\text{odd}}(z) &= \frac{\alpha^2 + \frac{k}{2}-\frac{1}{4}}{4k z^{1/2}(z-1)^{1/2}} \begin{pmatrix} -4z+2 & 2i \\ 2i & 4z-2 \end{pmatrix} + \frac{\left(k-\frac{1}{2}\right)i}{4 z^{1/2}(z-1)^{1/2}} \begin{pmatrix} -2\cos(y_{\alpha}) & -2i\cos(y_{\alpha+1})\\ -2i\cos(y_{\alpha-1}) & 2\cos(y_{\alpha}) \end{pmatrix}, \nonumber \\
	G_k^{\text{even}}(z) &= \frac{\alpha^2 + \frac{k}{2}-\frac{1}{4} }{4 k z^{1/2}(z-1)^{1/2}} \begin{pmatrix} 4\sqrt{z}\sqrt{z-1} & 0\\ 0 & 4z\sqrt{z}\sqrt{z-1} \end{pmatrix} + \frac{\left(k-\frac{1}{2}\right)i}{4 z^{1/2}(z-1)^{1/2}} \begin{pmatrix} -2\sin(y_{\alpha}) & -2i\sin(y_{\alpha+1})\\ -2i\sin(y_{\alpha-1}) & 2\sin(y_{\alpha}) \end{pmatrix}. \nonumber 
\end{align}
One should remark that $(-\varphi(z))^\alpha \neq \varphi(z)^\alpha(-1)^\alpha$ with standard branch cuts when deriving this formula. Also, we have used $\varphi(z) = \exp(i\arccos(2z-1))$ \cite[\S 5.0]{Vanlessen}.

The functions $y_\gamma$ above are dependent on $z$ by
\begin{align}
	y_\gamma = \gamma \left( \arccos(2z-1) - \pi \right), & \label{ygamma} \\
	y_{\gamma} \sim -2\gamma\sqrt{z} \sum_{j=0}^{\infty}\rho_{1,j} z^j, & \qquad \rho_{1,j}  = \frac{(1/2)_j}{j!(1+2j)}, \\%\rho_{1,j}  = \frac{(1/2)_j}{k!(1+2k)}, \\
	y_\gamma^k \sim (-2\gamma\sqrt{z})^k \sum_{j=0}^\infty \rho_{k,j} z^n, & \qquad \rho_{k,j} = \sum_{l=0}^j \rho_{k-1,l} \rho_{1,j-l}. \nonumber 
\end{align}
Note that with the standard branch cuts for the powers, $y_{\gamma}$ is real on the interval $[0,1]$.

We intend to expand all terms appearing in the definitions of $G_k$, starting with
\begin{align} %Alfredo/ sectie right monomial	\cos y_\gamma &\sim \sum_{j=0}^\infty H_{j,\gamma}^{\text{odd}} z^j \quad = 1 +\sum_{m=1}^\infty\left[ \sum_{j=1}^m (-1)^j (-2\gamma)^{2j} \frac{\rho_{2j,m-j}}{(2j)!}\right]z^m, \nonumber \\    \frac{\cos(y_{\gamma})}{2 z^{1/2}(z-1)^{1/2}} &\sim (-16z)^{-1/2} \sum_{m=0}^{\infty}\left[1+\sum_{j=1}^m {-1/2 \choose j} (-1)^{j} H^{\text{odd}}_{m-j,\gamma}\right]z^m,\\     H_{m-j,\gamma}^{\text{odd}} = \delta_{m-j} + \sum_{l=1}^{m-j} (-1)^l (-2\gamma)^{2l} \frac{\rho_{2l,m-j-l}}{(2l)!}
	\frac{\cos(y_{\gamma})}{2 z^{1/2}(z-1)^{1/2}} &\sim (-16z)^{-1/2} \sum_{m=0}^{\infty}\left[1+\sum_{j=1}^m {-1/2 \choose j} (-1)^{j} \left\{\delta_{0,m-j} + \sum_{l=1}^{m-j} \frac{(-1)^l}{(2l)!} (-2\gamma)^{2l} \rho_{2l,m-j-l}\right\}\right]z^m,\\ %Alfredo/ sectie right monomial		\sin y_\gamma &\sim -2\gamma\sqrt{z}\sum_{n=0}^\infty H_{n,\gamma}^{\text{even}} v^{n} \quad = -2\gamma\sqrt{z}\sum_{m=0}^{\infty} \left[\sum_{j=0}^m (-1)^j (-2\gamma)^{2j+1}\frac{\rho_{2j+1,m-j} }{(2j+1)!}\right]z^{m},\nonumber \\	\frac{\sin(y_{\gamma})}{2 z^{1/2}(z-1)^{1/2}} &\sim \frac{\gamma i}{2} \sum_{m=0}^{\infty}\left[\sum_{j=0}^m {-1/2 \choose j} (- 1)^{-j}H^{\text{even}}_{m-j,\gamma}\right]z^m.  \\   H_{m-j,\gamma}^{\text{even}} = \sum_{l=0}^{m-j} (-1)^l (-2\gamma)^{2l+1}\frac{\rho_{2l+1,m-j-l} }{(2l+1)!}
	\frac{\sin(y_{\gamma})}{2 z^{1/2}(z-1)^{1/2}} &\sim \frac{\gamma i}{2} \sum_{m=0}^{\infty}\left[\sum_{j=0}^m {-1/2 \choose j} (-1)^{-j} \left\{\sum_{l=0}^{m-j} \frac{(-1)^l}{(2l+1)!} (-2\gamma)^{2l+1} \rho_{2l+1,m-j-l} \right\} \right]z^m,
\end{align}
where we have used the Kronecker delta $\delta_{i,j}$. %Toegevoegd
Also, 
\[
\frac{4z-2}{4z^{1/2}(z-1)^{1/2}} \sim \frac{i}{2\sqrt{z}} + \frac{i}{2\sqrt{z}} \sum_{n=1}^{\infty} \left(2{-1/2 \choose n-1} +{-1/2 \choose n}\right)(-1)^{n} z^n. 
\]

We still need to expand the power $\bar{\phi}_n(z)^{-k}$ in \cref{EDeltak_Tk} as $z \to 0$. To that end, we can combine \eqref{Exin}, \eqref{Ephibar} and \eqref{EHn}. It is quite standard, but increasingly tedious, for series expansions to involve convolutions whose coefficients can be found recursively. That is the origin of the coefficients $f_j$ and $g_{k,m}$ below, and we will use this pattern several times more in the remainder of this section:
\begin{align} 
	\bar{\phi}_n(z) \sim  \theta(z) i \sqrt{z} \sum_{j=0}^\infty f_j z^j, & \qquad f_j = -\frac{(1/2)_j }{j!(1+2j)} - \frac{1}{2 m A_m}\sum_{k=0}^{\min(m-1,j)} (-1)^{j-k} {1/2 \choose j-k} A_{m-k-1}, \label{Efphibar} \\ 
	(\bar{\phi}_n(z))^{-1} \sim  (-z)^{-1/2} \sum_{m=0}^{\infty} g_{1,m} z^{m}, & \qquad g_{1,0} = \frac{1}{f_0}, \qquad \qquad \qquad \quad g_{1,m} = \frac{-1}{f_0} \sum_{j=0}^{m-1} g_{1,j} f_{m-j}, \\
	(\bar{\phi}_n(z))^{-k} \sim  (-z)^{-k/2} \sum_{m=0}^{\infty} g_{k,m} z^{m}, & \qquad g_{k,m} = \sum_{l=0}^n g_{k-1,l} g_{1,m-l}.
\end{align}

The coefficients $W_{k,i}^{\L}$ in expansion \eqref{Wkm} for the functions $s_k^{\L}(z)$ are given by 
\begin{align}
         W_{k,i}^{\L} & = \frac{(\alpha,k-1)}{-(-1)^{\lceil k/2 \rceil +1}4^k}  2^{-\alpha\sigma_3}\sum_{j=0}^{i+(k+1)/2} g_{k,j} G_{k,i+(k+1)/2-j}^{\operatorname{odd} }  2^{\alpha\sigma_3}, \nonumber \\ 
         W_{k,i}^{\L} & = \frac{(\alpha,k-1)}{-(-1)^{\lceil k/2 \rceil +1}4^k}  2^{-\alpha\sigma_3}\left( \sum_{j=0}^{i+k/2} g_{k,j} G_{k,i+k/2-j}^{\operatorname{even}} \right) 2^{\alpha\sigma_3} - \frac{(\alpha,k-1)(4\alpha^2+2k-1)g_{k,i+k/2}}{2 k 4^k}I, \nonumber 
\end{align}
respectively for odd and even $k$. The $V_{k,i}^{\L}$ can be obtained by leaving out the term with $I$.

\subsection{Right disk with monomial $Q(x)$} \label{SexplR}

Unlike in the Jacobi case, the expressions for the left and right disks are not symmetric, since they correspond to qualitatively different behaviour of the polynomials near a hard edge and near a soft edge. With $w = z-1$ and $\tau_\gamma = \gamma \left(\arccos(2w+1) \right)$, we have
\begin{align}
	\Delta_k^{\R}(z) = & \frac{2^{-\alpha\sigma_3}}{2\left(- \xi_n(z) \right)^k} \Omega_k(z) 2^{\alpha\sigma_3}, \nonumber \\
	\Omega_k^{\text{odd}}(z) = & \frac{\nu_k}{4\sqrt{w}\sqrt{w+1}}\begin{pmatrix} -4w-2 & 2i \\ 2i & 4w+2 \end{pmatrix} + \frac{-6k\nu_k}{4\sqrt{w}\sqrt{w+1}}  \begin{pmatrix} -2\cos(\tau_{\alpha}) & 2i\cos(\tau_{\alpha+1})\\ 2i\cos(\tau_{\alpha-1}) & 2\cos(\tau_{\alpha}) \end{pmatrix},\nonumber \\
	\Omega_k^{\text{even}}(z) = & \nu_k I + \frac{-6k\nu_k}{4\sqrt{w}\sqrt{w+1}}  \begin{pmatrix} -2i\sin(\tau_{\alpha}) & -2\sin(\tau_{\alpha+1})\\ -2\sin(\tau_{\alpha-1}) & 2i\sin(\tau_{\alpha}) \end{pmatrix}. \nonumber
\end{align}
For the expansion of $\Omega_k(z)$, we observe that
\begin{align}
	\arccos(1+2w) \sim 2\sqrt{-w} \sum_{j=0}^\infty \rho_{1,j}w^j, & \qquad \rho_{1,j} = \frac{(1/2)_j(-1)^j}{(1+2j)j!}, \nonumber \\
	\left[\gamma\arccos(1+2w)\right]^m \sim (2\gamma i)^m w^{m/2} \sum_{j=0}^\infty \rho_{m,j} w^j, & \qquad \rho_{m,j} = \sum_{l=0}^j \rho_{m-1,l}\rho_{1,j-l}, \nonumber \\%Alfredo/ 	\cos\left(\tau_\gamma\right) \sim \sum_{m=0}^\infty \text{\foreignlanguage{russian}{ю}}_{\gamma,m} w^m, & \qquad \text{\foreignlanguage{russian}{ю}}_{\gamma,m} = \sum_{j=0}^m (-1)^j(2\gamma i)^{2j} \rho_{2j,m-j}/(2j!), \nonumber \\   	\frac{\cos\left(\tau_\gamma\right)}{\sqrt{w^2+w}} \sim w^{-1/2} \sum_{m=0}^\infty \text{\foreignlanguage{russian}{б}}_{\gamma,m} w^m, & \qquad \text{\foreignlanguage{russian}{б}}_{\gamma,m} = \sum_{j=0}^m {-1/2 \choose j} \text{\foreignlanguage{russian}{ю}}_{\gamma,m-j},  \nonumber \\  \text{\foreignlanguage{russian}{ю}}_{\gamma,m-j} = \sum_{l=0}^{m-j} (-1)^l(2\gamma i)^{2l} \rho_{2l,m-j-l}/(2l!), \nonumber \\ 
	\frac{\cos\left(\tau_\gamma\right)}{\sqrt{w^2+w}} \sim w^{-1/2} \sum_{m=0}^\infty \text{\foreignlanguage{russian}{б}}_{\gamma,m} w^m, & \qquad \text{\foreignlanguage{russian}{б}}_{\gamma,m} = \sum_{j=0}^m {-1/2 \choose j} \left[\sum_{l=0}^{m-j} \frac{(-1)^l}{(2l)!}(2\gamma i)^{2l} \rho_{2l,m-j-l}\right],  \nonumber \\%Alfredo/ 	\sin\left(\tau_\gamma\right) \sim \sqrt{w} \sum_{m=0}^\infty \text{\foreignlanguage{russian}{я}}_{\gamma,m} w^m, & \qquad \text{\foreignlanguage{russian}{я}}_{\gamma,m} = \sum_{j=0}^m (-1)^j(2\gamma i)^{2j+1} \rho_{2j+1,m-j}/(2j+1)!, \nonumber \\	\frac{\sin\left(\tau_\gamma\right)}{\sqrt{w^2+w}} \sim \sum_{m=0}^\infty \text{\foreignlanguage{russian}{ч}}_{\gamma,m} w^m, & \qquad \text{\foreignlanguage{russian}{ч}}_{\gamma,m} = \sum_{j=0}^m {-1/2 \choose j} \text{\foreignlanguage{russian}{я}}_{\gamma,m-j}. \nonumber % \text{\foreignlanguage{russian}{я}}_{\gamma,m-j} = \sum_{l=0}^{m-j} (-1)^l(2\gamma i)^{2l+1} \rho_{2l+1,m-j-l}/(2l+1)!, \nonumber \\
	\frac{\sin\left(\tau_\gamma\right)}{\sqrt{w^2+w}} \sim \sum_{m=0}^\infty \text{\foreignlanguage{russian}{ч}}_{\gamma,m} w^m, & \qquad \text{\foreignlanguage{russian}{ч}}_{\gamma,m} = \sum_{j=0}^m {-1/2 \choose j} \left[\sum_{l=0}^{m-j} \frac{(-1)^l}{(2l+1)!}(2\gamma i)^{2l+1} \rho_{2l+1,m-j-l}\right]. \nonumber 
\end{align}

With these expressions in hand, we focus again on the phase function. We construct the power series of $\xi_n(1+w)^{-k}$ as follows:
\begin{align}
	[\sqrt{1+w}-1]^m &\sim w^m \sum_{l=0}^\infty u_{m,l} w^l, \quad u_{1,l} = {1/2 \choose l}, \quad u_{m,l} = \sum_{j=0}^l u_{m-1,j} u_{1,l-j}, \nonumber \\ 
	\sqrt{2-2\sqrt{1+w}} &\sim \sqrt{-2w} \sum_{k=0}^\infty r_k w^k, \quad v_{1,j} = {1/2 \choose j+2}, \quad v_{m,l} = \sum_{j=0}^l v_{m-1,j} v_{1,l-j}, \quad r_k = \sum_{l=0}^k {1/2 \choose k-l} v_{k-l,l} 2^{k-l}, \nonumber \\ 
	q_m &= \sum_{l=0}^m \frac{(1/2)_{m-l} u_{m-l,l} }{(-2)^{m-l}((m-l)! (1+2m-2l)}, \nonumber \\ 
	\xi_n(1+w) &\sim \sqrt{w} \sum_{j=1}^\infty f_j w^j, \nonumber \\
	f_j & = 2\sum_{l=0}^j q_l r_{j-l} - \frac{1}{m A_m}\sum_{k=0}^{\min(m-1,j)} (-1)^{m-k-1} {1/2 \choose j-k} \frac{\Gamma(-1/2-k)}{\Gamma(1/2-m)\Gamma(m-k)}  \label{Efxin} \\
	\left(-\xi_n(1+w)\right)^{-1} &\sim w^{-3/2} \sum_{m=0}^{\infty} g_{1,m} z^{m}, \qquad g_{1,0} = \frac{-1}{f_1}, \qquad \qquad g_{1,m} = \frac{-1}{f_1} \sum_{j=1}^{m} f_{j+1} g_{1,m-j}, \nonumber \\
	\left(-\xi_n(1+w)\right)^{-k} &\sim w^{-3k/2} \sum_{m=0}^{\infty} g_{k,m} w^m, \qquad g_{k,m} = \sum_{l=0}^n g_{k-1,l} g_{1,m-l}. \nonumber
\end{align}
The phase function $\xi_n(z)$ should be $\mathcal{O}\left( (z-1)^{3/2} \right)$ \cite[Rem. 3.22]{Vanlessen} and \cref{Efxin} indeed indicates that $f_0 = -2+2q_0r_0$ is zero, so we have started the indices in the expansion of $\xi_n(z)$ from $j=1$. Keeping in mind the Kronecker delta $\delta_{i,j}$ and that ${-1/2 \choose -1} = 0$, we arrive at the final result
\begin{align}
	\Omega_{k,j}^{\text{odd}} &= \frac{\nu_k}{2} \begin{pmatrix} - {-1/2 \choose j} -2 {-1/2 \choose j-1} & i {-1/2 \choose j} \\ i {-1/2 \choose j} & {-1/2 \choose j} +2 {-1/2 \choose j-1} \end{pmatrix} -3k\nu_k \begin{pmatrix} \text{\foreignlanguage{russian}{б}}_{\alpha,j} & i\text{\foreignlanguage{russian}{б}}_{\alpha+1,j} \\ i\text{\foreignlanguage{russian}{б}}_{\alpha-1,j} & \text{\foreignlanguage{russian}{б}}_{\alpha,j} \end{pmatrix}, \nonumber \\
	\Omega_{k,j}^{\text{even}} &= \nu_k I \delta_{0,j} -3k\nu_k \begin{pmatrix} -i\text{\foreignlanguage{russian}{ч}}_{\alpha,j} & -\text{\foreignlanguage{russian}{ч}}_{\alpha+1,j} \\ -\text{\foreignlanguage{russian}{ч}}_{\alpha-1,j} & i\text{\foreignlanguage{russian}{ч}}_{\alpha,j} \end{pmatrix}, \nonumber \\
	W_{k,i}^{\R} &= \frac{2^{-\alpha\sigma_3}}{2} \left(\sum_{j=0}^{i+\lceil 3k/2 \rceil}  \Omega_{k,j} g_{k,i+\lceil 3k/2 \rceil-j} \right) 2^{\alpha\sigma_3}, \nonumber 
\end{align}
to which we add $-\nu_k g_{k,i+3k/2}I$ when $k$ is even.

\subsection{Left disk with general polynomial $Q(x)$} \label{SprecLnonmon}

We can re-use the expansion of $G_k(z)$. However, in this case $H_n(z)$ has an expansion in $n$, given by \eqref{EHn_nonmon}. Hence, we have to determine the expansion of $\bar{\phi}_n(z)$ \cref{Ephibar} near $z=0$ anew, but afterwards one continues as in \S \ref{SexplLeft} to get the $W_{k,i,l}^{\L}$ from \cref{Wkmnonmon}:
\begin{align}
	\bar{\phi}_n(z) &\sim -\theta(z) i \sqrt{z} \sum_{l=0}^\infty \sum_{j=0}^\infty f_j^l z^j n^{-l/m}, \nonumber \\ 
	f_j^0 &= \frac{(1/2)_j }{j!(1+2j)} + \frac{1}{2mA_m}\sum_{k=0}^{\min(m-1,j)} (-1)^{j-k} {1/2 \choose j-k} A_{m-k-1}, \nonumber \\
	f_j^l &= \frac{1}{2} \sum_{i=0}^{\min(j,m-1)} (-1)^{j-i} {1/2 \choose j-i} \sum_{n=\max(i+1,m-l)}^m  q_n A_{n-i-1} \beta^{n,n+l-m}, \qquad l > 0, \nonumber \\
	(\bar{\phi}_n(z))^{-1} &\sim (-z)^{-1/2} \sum_{l=0}^\infty \sum_{j=0}^{\infty} g_{1,j}^l z^j n^{-l/m}, \qquad g_{1,0}^0 = \frac{1}{f_0}, \qquad g_{1,n}^0 = \frac{-1}{f_0} \sum_{i=0}^{n-1} g_{1,i}^0 f_{n-i}^0,  \nonumber \\
	g_{1,i}^l &= \sum_{y=0}^i \sum_{p=0}^{i-y} g_{1,i-y-p}^0 \sum_{q=0}^{l-1}  g_{1,y}^q f_p^{l-q}, \qquad g_{k,j}^0 = \sum_{l=0}^j g_{k-1,l}^0 g_{1,j-l}, \nonumber \\
	(\bar{\phi}_n(z))^{-k} &\sim (-z)^{-k/2} \sum_{l=0}^\infty \sum_{j=0}^{\infty} g_{k,j}^l z^j n^{-l/m}, \qquad g_{k,n}^l = \sum_{j=0}^l \sum_{i=0}^n g_{k-1,i}^{j} g_{1,n-i}^{l-j}. \nonumber 
\end{align}
This gives a pole of order $\mathcal{O}(z^{-\lceil k/2 \rceil})$ in $\Delta^{\L}_k(z)$, as it should \cite[Rem 3.29]{Vanlessen}.

\subsection{Right disk with general polynomial $Q(x)$} \label{SprecRnonmon}

We can also reuse the expansion of $\Omega_k(z)$ and the expansion of $\xi_n(z)$ up to the definition of $f_j$ (so the expansion of $\arccos(\sqrt{z})$) from \S \ref{SexplR}. Expanding \eqref{Exin}, we get
\begin{align}
	\xi_n(1+w) &\sim \sqrt{w} \sum_{l=0}^\infty \sum_{j=1}^\infty f_j^l w^j  n^{-l/m}, \\ 
	f_j^0 &= \left[2\sum_{l=0}^{\min(j,m-1)} q_l r_{j-l}\right] -\frac{1}{m A_m}\left[\sum_{i=0}^{\min(j,m-1)} {1/2 \choose j-i} (-1)^{m-i-1} \frac{\Gamma(-1/2-i)}{\Gamma(-1/2-m)\Gamma(m-i)}  \right], \nonumber \\
	f_j^l &= \frac{-1}{2} \sum_{k=0}^{m-1} \left[ \sum_{z=\max(k+1,m-l)}^m  q_z A_{z-k-1} \beta^{z,z+l-m} \right] \left\{\sum_{i=0}^{\min(k,j)} {1/2 \choose j-i}  {k \choose i}\right\}. \label{EfxinNonmon}
\end{align}
It follows by the construction in \cref{Smrs} that $f_0^l =0$. Thus, $\left(-\xi_n(1+w) \right)^{-k}$ again has a pole of order $\mathcal{O}(w^{- 3k/2 })$:
\begin{align}
	(-\xi_n(1+w))^{-1} \sim w^{-3/2} \sum_{l=0}^\infty \sum_{j=0}^\infty g_{1,j}^l w^j  n^{-l/m}, & \qquad g_{1,0}^0 = \frac{-1}{f_1^{0}}, \qquad g_{1,i}^0 = \frac{-1}{f_1} \sum_{j=1}^i f_{j+1}^{0} g_{1,i-j}^{0}, \\
	g_{1,i}^l = \sum_{y=0}^{i-1} \sum_{p=1}^{i-y} g_{1,i-y-p}^0 \sum_{q=0}^{l-1} g_{1,y}^q f_{p+1}^{l-q}, & \qquad g_{k,j}^0 = \sum_{l=0}^j g_{k-1,l}^0 g_{1,j-l} , \nonumber \\
	\left(-\xi_n(1+w) \right)^{-k} \sim w^{-3k/2} \sum_{l=0}^\infty \sum_{j=0}^{\infty} g_{k,j}^l w^j n^{-l/m}, & \qquad g_{k,n}^l = \sum_{j=0}^l \sum_{i=0}^n g_{k-1,i}^{j} g_{1,n-i}^{l-j}. \nonumber 
\end{align}

\subsection{Left disk with general function $Q(x)$}
For general $Q$, the MRS number $\beta_n$ is dependent on $n$ in a way that is not easy to predict, see for example \cref{EbetaExp}. As a result, so is $V_n(x)$. This means that the series expansions that form the result of this section are also dependent on $n$. Hence, strictly speaking, they are not the true asymptotic expansions. However, for any given $n$ they can still be useful in computations and give a computational time independent of $n$. We proceed by expanding the function $h_n$ in a Taylor series, using contour integrals for the coefficients (see also \cref{SmrsNonpoly}):
\begin{equation} \label{Edn}
	h_n(z) \sim \sum_{l=0}^\infty d_l(n) z^l, \qquad d_l(n) = \frac{1}{2\pi i} \oint_{\Gamma_z} \frac{\sqrt{y}V_n^\prime(y) dy}{\sqrt{y-1}y^{l+1}},
\end{equation}
Continuing the analysis as before, we find an expansion for $\bar{\phi}_n$,

\begin{align}
	\sqrt{y^2-1} h_n(y^2) &\sim i\left[\sum_{p=0}^\infty (-1)^p{ 1/2 \choose p } y^{2p}\right] \left[\sum_{l=0}^\infty d_l(n) y^{2l}\right]  \sim i\sum_{k=0}^\infty \text{\foreignlanguage{russian}{ц}}_k y^{2k} \nonumber \\
	\bar{\phi}_n(z) &= \frac{-1}{2} \int_1^z \frac{\sqrt{y-1}}{\sqrt{y}} h_n(y) dy -\frac{\pi i}{2}  = \frac{-1}{2} \int_0^{\sqrt{z}} \sqrt{y^2-1}  h_n(y^2) dy, \nonumber \\
	\bar{\phi}_n(z) &\sim \frac{-1}{2} \sum_{l=0}^\infty \frac{\sqrt{z}^{l+1}}{(l+1)!} \left. \frac{\partial^l i\sum_{k=0}^\infty \text{\foreignlanguage{russian}{ц}}_k x^{2k} }{\partial x^l} \right|_{x=0} \sim \frac{-i\sqrt{z} h_n(0)}{2} - \frac{z}{4} \left[\frac{2\sqrt{z} h_n(z)}{\sqrt{z-1}} + \sqrt{z-1} h_n'(z)2\sqrt{z} \right] + \cdots, \nonumber \\
	\bar{\phi}_n(z) &\sim -\theta(z) i \sqrt{z} \sum_{j=0}^\infty f_j(n) z^j, \qquad f_j(n) = \frac{-1}{2(2j+1)} \sum_{l=0}^j (-1)^l { 1/2 \choose l } d_{j-l}(n), \nonumber
\end{align}
where we used \cref{Ephibar} and the power series of an integral. We plug this into \cref{Efphibar} and continue with the other equations in \cref{SexplLeft} to get the expansion of $s_k^{\L}(z)$ for general $Q(x)$. The resulting matrices are also treated as for monomial $Q(x)$ in \cref{ScompU,Ssimpl} to obtain the $U$- and $Q$-matrices.%Added 22 dec

\subsection{Right disk with general function $Q(x)$}
We proceed as in the previous section by expanding $h_n(z)$ using contour integrals:
\begin{equation} \label{Ecn}
	h_n(z) \sim \sum_{l=0}^\infty c_l(n) (z-1)^l, \qquad c_l(n) = \frac{1}{2\pi i} \oint_{\Gamma_z} \frac{\sqrt{y}V_n^\prime(y) dy}{\sqrt{y-1}(y-1)^{l+1}}.
\end{equation}
We can again continue with the phase function:
\begin{align}
	\xi_n(z) = \frac{-1}{2} \int_1^z \frac{\sqrt{y-1}}{\sqrt{y}} h_n(y) dy  & \quad = \frac{-1}{2} \int_0^{\sqrt{w}} \frac{t}{\sqrt{t^2+1}} h_n(t^2+1) 2tdt, \nonumber \\
	\frac{t^2}{\sqrt{t^2+1}} h_n(t^2+1) \sim t^2 \sum_{j=0}^\infty \text{\foreignlanguage{russian}{ч}}_j t^{2j}, & \qquad \text{\foreignlanguage{russian}{ч}}_j = \sum_{l=0}^{j} {-1/2 \choose l} c_{j-l}(n), \nonumber \\
	\xi_n(1+w) \sim \sqrt{w} \sum_{j=1}^\infty f_j(n) w^j, & \qquad f_j(n) = \frac{-1}{2j+1} \sum_{l=0}^{j-1} {-1/2 \choose l} c_{j-1-l}(n). \nonumber
\end{align}

\begin{remark} \label{RgenCont}
Remark that in order to use these expressions, one only needs to (numerically) compute $\beta_n$ and the contour integrals. In contrast, we need to define $m$ times more coefficients in the case $Q(x)$ is a general polynomial to compute \cref{EfxinNonmon}, even though we can also just use those same contour integrals. A brief numerical comparison is given in \cref{SnumerGen}, and we can note that the expressions that we have specifically derived for general polynomial $Q$ are fully explicit.
\end{remark}

\section{Examples and numerical results} \label{Snum}

\subsection{Monomial $Q(x)$}

A case of specific interest in the context of Gaussian quadrature is the standard Laguerre polynomial. We illustrate the accuracy of the asymptotic expansion in the left disk using our simplifications (\cref{E:simplifiedjump} and \cref{Tsimpl}) in the left part of \cref{Fmon}. The values we compare with are computed using a recurrence relation for orthonormal polynomials with exact coefficients with calculations in double precision. We evaluate at a point close to the normalized origin. The errors decrease as $\mathcal{O}(n^{-T})$ with $T$ the number of terms as expected. For small $n$, the expansions may diverge with increasing $T$ and the errors saturate at about $10^{-14}$. %Corrected 22dec: We evaluate at a point in the lower complex plane close to the normalized origin.

\begin{figure}[h]
\centering
%\subfloat{ \includegraphics[width = 0.49\textwidth]{fig/stdlag1em3LeftSameSize} }
%\subfloat{ \includegraphics[width = 0.49\textwidth]{fig/wei5Left1pio100BigFloatSamesize} } 
\subfloat{ \includegraphics[width = 0.49\textwidth]{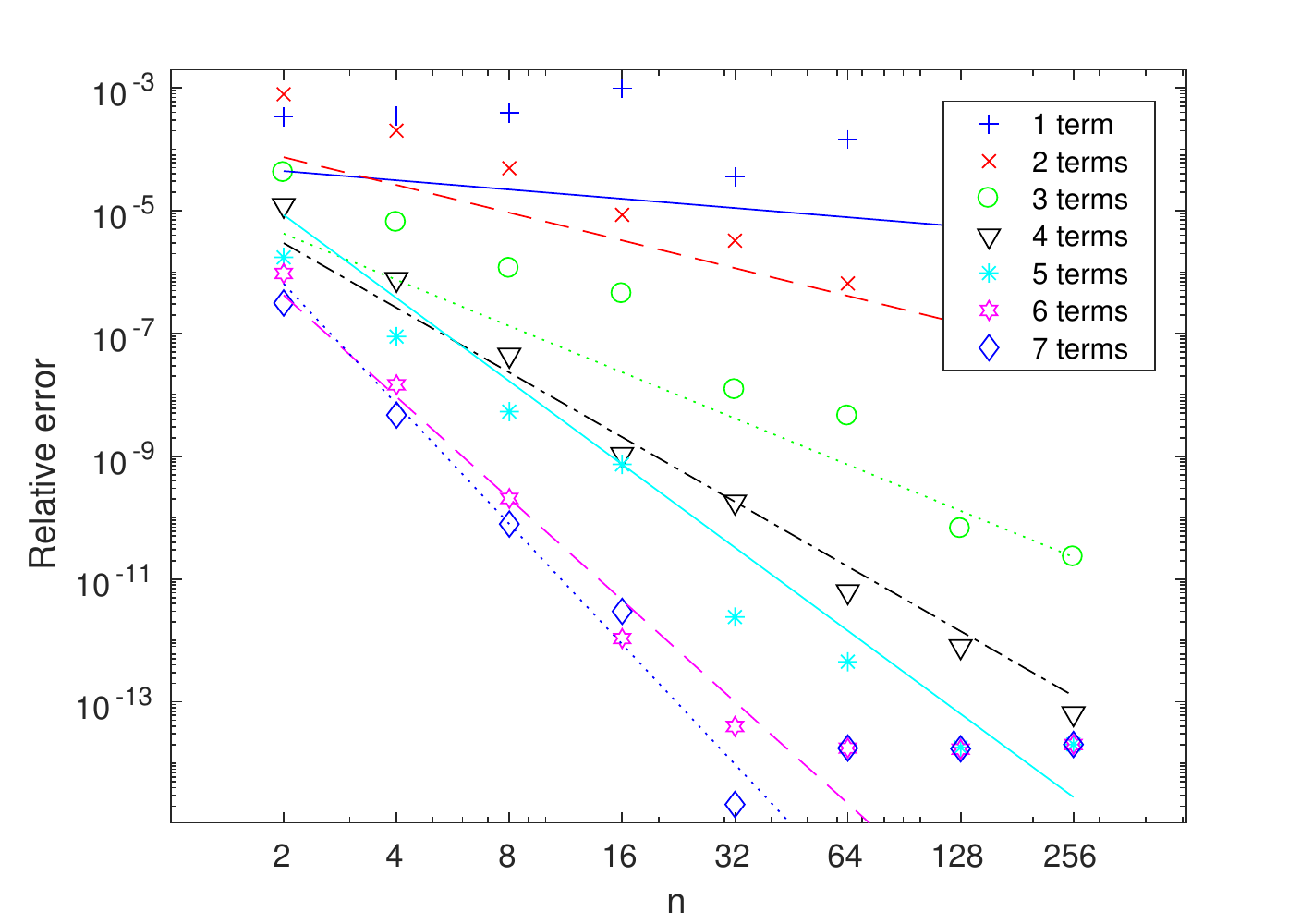} } %For arXiv
\subfloat{ \includegraphics[width = 0.49\textwidth]{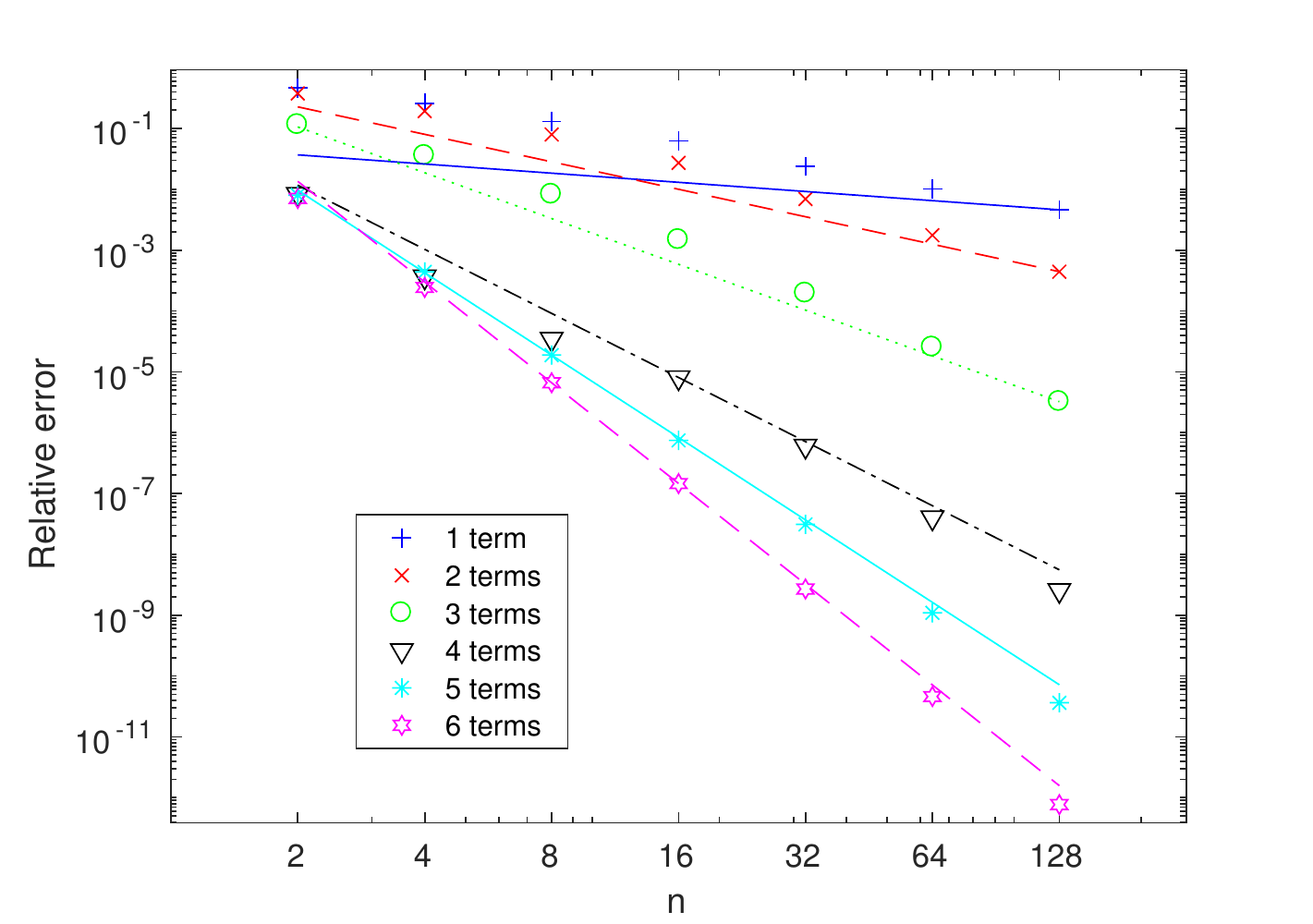} } % For arXiv
\caption{Relative error of the asymptotic expansion in the left boundary region as a function of the degree, for $w(x) = e^{-x}, x=4n/1000$ (left) and $w(x) = x^{2.8} \exp(-0.7x^3-3/2), x = \beta_n(1 - i)/100$ (right). }
\label{Fmon}
\end{figure}
%Alfredo/ misschien toen bibtex nog niet gerund. Marcus suggereerde eerst "In the continuous Lanczos algorithm \cite[\S 36]{trefBau}" maar Algo 37.1 lijkt me specifieker
In the right part of \cref{Fmon}, we show results for another monomial $Q(x)$, where the higher order terms are now calculated with \cref{ERrlseries} where the summation index $n$ ranges from $0$ to $11$. Here, we have used high-precision arithmetic to compute the reference solution using standard methods. In the continuous Lanczos algorithm \cite[Algo 37.1]{trefBau} for the computation of the recurrence coefficients, we have to evaluate integrals such as $a_n = \int_0^\infty x^{\alpha+1} \exp(-Q(x)) p_{n-1}^2(x) dx$. In order to obtain sufficiently accurate `exact' results using the recurrence relation, we had to evaluate the recurrence coefficients with 26 digits of accuracy, a computation that we performed in {\sc Julia}. All computations with the asymptotic expansions were performed in standard floating point double precision. Having said that, the errors again decrease like $\mathcal{O}(n^{-T})$ as we expect, hence we conclude that the higher order terms are computed correctly. The asymptotic expansions 
of the coefficients $\gamma_n$, $\alpha_n$ and $\beta_n$, and of the polynomials in the other regions and for other values of $x$ exhibit similar behaviour.

\subsection{Connection with Hermite polynomials} \label{Sherm}

\cite[18.7.17]{DLMF} states that $H_{2n}(x) = (-1)^n 2^{2n}n!L_n^{(-1/2)}(x^2)$, but as these easily overflow numerically, we construct normalized Hermite polynomials by
\begin{equation}
	H_0^{\text{norm}}(x) = \pi^{-1/4}, \quad H_{-1}^{\text{norm}}(x) = 0, \quad H_j^{\text{norm}}(x) = \frac{2x H_{j-1}^{\text{norm}}(x)}{\sqrt{2j}} -\frac{(j-1)H_{j-2}^{\text{norm}}(x)}{\sqrt{j(j-1)} }. \nonumber
\end{equation}	
As $(2n)!\sqrt{\pi} = n! 4^n \Gamma(n+1/2)$, we have that $H_{2n}^{\text{norm}}(x) = L_n^{(-1/2),\text{norm}}(x^2)$, the normalized associated Laguerre polynomial with positive leading coefficient. In the left part of \cref{Fherm}, we see that the asymptotic expansion in the right disk \cref{Epiboun} (for the summation index $n$ from $0$ to $10$) converges as expected to $H_{2n}^{\text{norm}}(x)$ as a function of $n$, evaluated at $x^2 = 0.97(4n)$ using \cref{ERrlseries}, $Q(x) = x$ and $\alpha = -1/2$.

\begin{figure}[h]
\centering
\subfloat{ \includegraphics[width = 0.49\textwidth]{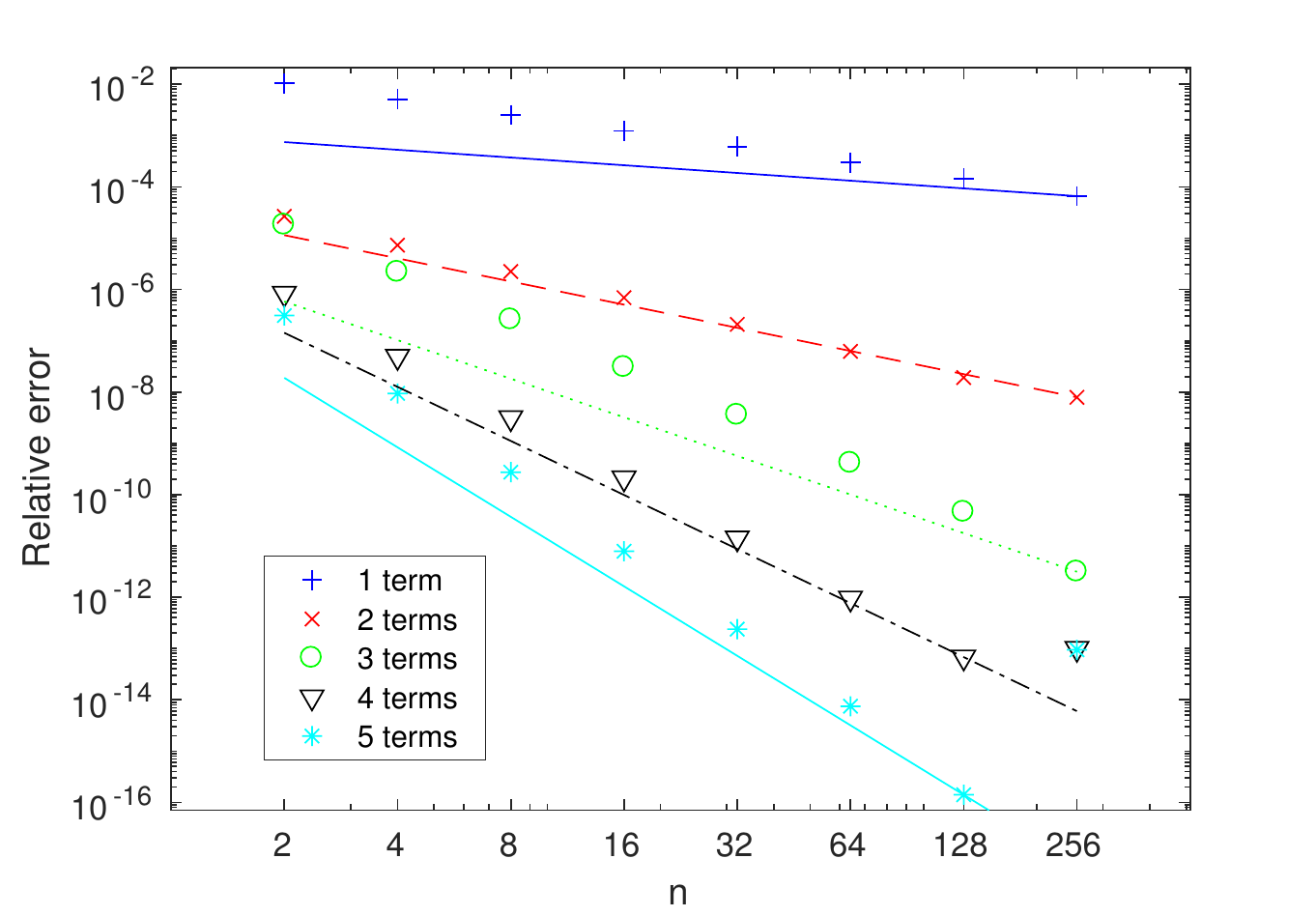} }%normHermEvenSameSize, for arXiv
\subfloat{ \includegraphics[width = 0.49\textwidth]{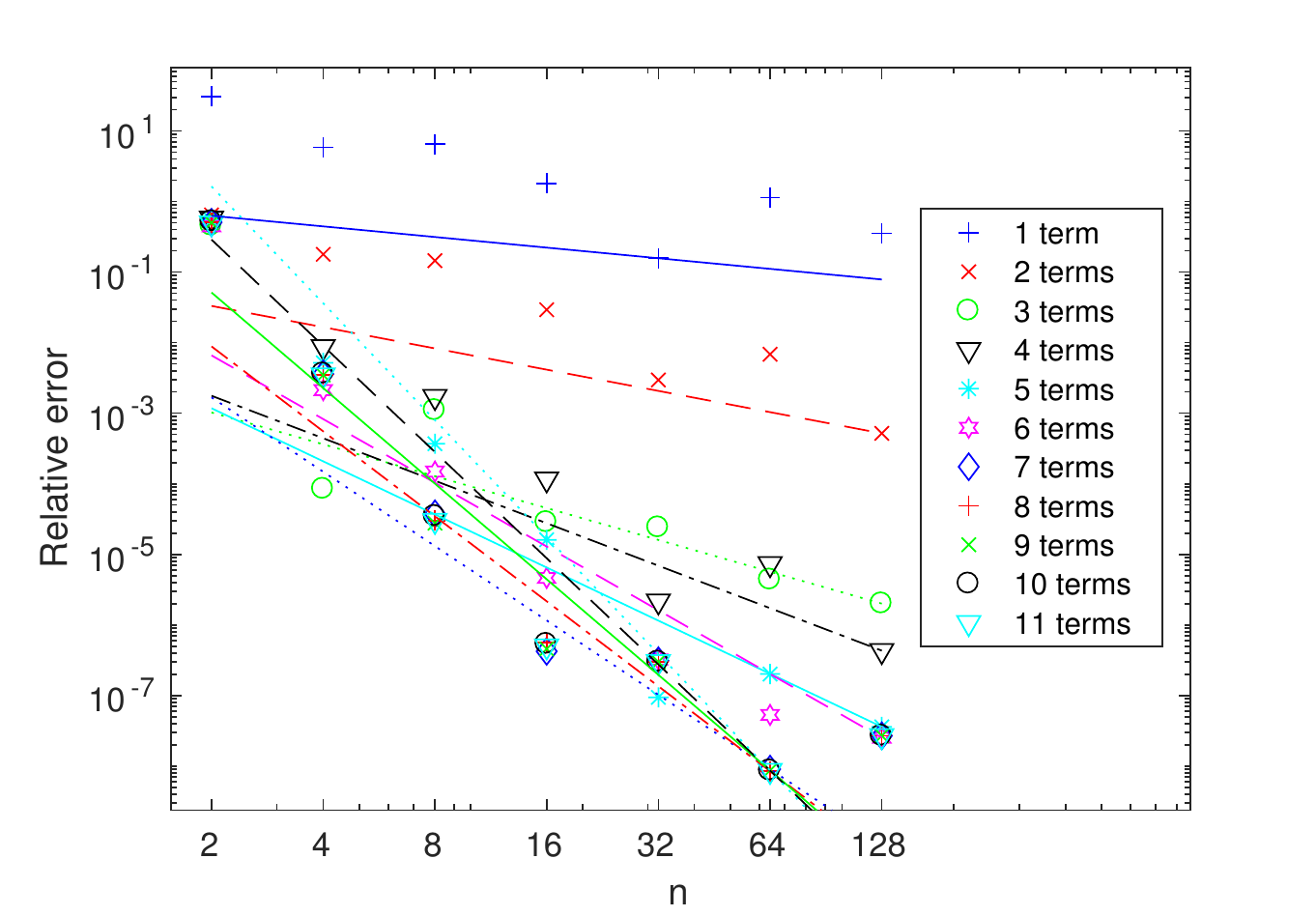} } %hermqm31Bul31new, for arXiv
%\subfloat{ \includegraphics[width = 0.49\textwidth]{fig/normHermEvenSSmin} }%normHermEvenSameSize
%\subfloat{ \includegraphics[width = 0.49\textwidth]{fig/hermqm31Bul31SSmin} } %hermqm31Bul31new
\caption{Relative error on the normalized Hermite polynomials as a function of the degree of the associated Laguerre polynomial $n$, for $w(x) = \exp(-x^2), H_{2n}(x), x=\sqrt{3.88n}$ (left) and $w(x) = \exp(-x^4+3x^2), H_{2n+1}(x), x = \sqrt{0.31\beta_n}$ (right). All calculations (also the recurrence coefficients) were performed in double precision.}
\label{Fherm}
\end{figure}
For odd degrees, a similar reasoning gives $H_{2n+1}^{\text{norm}}(x) = x L_n^{(1/2),\text{norm}}(x^2)$. Now, we explore the connection of a generalized weight $\exp(-x^4+3x^2)$ on $(-\infty, \infty)$ with a weight $\exp(-x^2+3x)$ on $[0, \infty)$, using the expansion in the lens \cref{EpiInt}. Although the right panel of \cref{Fherm} shows higher errors, we do get the $\mathcal{O}(n^{-T/m})$ convergence we expect, with $T$ the number of terms and $m=2$. It also illustrates that taking more terms is not always advantageous, as the asymptotic expansions diverge when increasing $T$ for a fixed $n$. This effect is more pronounced when $n$ is low. 

In both cases, $\alpha^2 = 1/4$ and the expansion in the Bessel region exhibits trigonometric behaviour as in the lens. Unlike in the Jacobi case \cite[\S 2.6]{jacobi}, they are not exactly equal for the same number of terms, but they agree more and more if $n$ and/or $T$ increase(s). However, the computation of higher order terms can be improved in this case as mentioned in \cref{Rspec}.

One could also go through \cite{DKMVZ} to obtain asymptotics of Hermite-type polynomials, and there are indeed many analogies between both approaches, as \cite{Vanlessen} was inspired by it. One advantage of exploiting the connection with Laguerre-type polynomials is that the $U_{k,m}^{\L}$ matrices are zero so their computations can be omitted, while the other approach computes $U$-matrices near both soft edges when straightforwardly implemented in an analogous way.
%In this respect, it would also be interesting to know whether there is a connection between polynomials with weights $x^\alpha \exp(-Q(x))$ on $(-\infty, \infty)$ and for example $x^{\alpha \pm 1/2} \exp(-Q(x))$ on $[0, \infty)$ to generalize known results with $\alpha=0$ or varying weights \cite[\S 1]{Vanlessen}.%Alfredo/ Niet direct gevonden in artikels Chihara/Arno Kuijlaars/Tom Claeys, evt vragen 

\subsection{General function $Q(x)$}

In this section, we provide numerical results for our claim that the expansions can also be used for general functions $Q(x)$ for the case $Q(x) = \exp(x)$. We have been able to verify that $\int_0^\beta \exp(x)\sqrt{x/(\beta-x)} dx/2/\pi$ agrees numerically with $\beta\exp(\beta/2) [I_0(\beta/2)+I_1(\beta/2)] /4$, see \cref{EexponQbeta}, as long as these do not overflow. 
Additionally, the corresponding expansion of $\beta_n$ \cref{EbetaExp} converges with the expected rate, as can be seen in \cref{Fexp}. %Toegevoegd: evt toe te voegen dat betas gespecifieerd en n berekend adhv int_0^beta... dx zie ComputePlotsFromArticle.m % The accuracy of the asymptotics of the recurrence coefficients $b_{n-1}$ associated to this MRS number is shown in \cref{Fexp}. 
We also show that we can approximate this special orthonormal polynomial in the bulk of its spectrum using \cref{EpiInt} in the right side of \cref{Fexp}.
The reference results were again computed using recurrence coefficients with 26 digits. The errors agree with the expected orders which are only negative integer powers of $n$. This is because other types of dependencies on $n$ arising from $\beta_n$ (e.g. the $\mathcal{O}(n^{-1/m})$ for general polynomial $Q(x)$) were eliminated by for each $n$ numerically computing $\beta_n$ and the contour integrals \cref{Edn,Ecn}. %Finally, we show that we can approximate this special orthonormal polynomial in the bulk of its spectrum using \cref{EpiInt} in the right side of \cref{Fexp}.

\begin{figure}[h]
\centering
%\subfloat{ \includegraphics[width = 0.49\textwidth]{fig/bnExpnew} }
\subfloat{ \includegraphics[width = 0.49\textwidth]{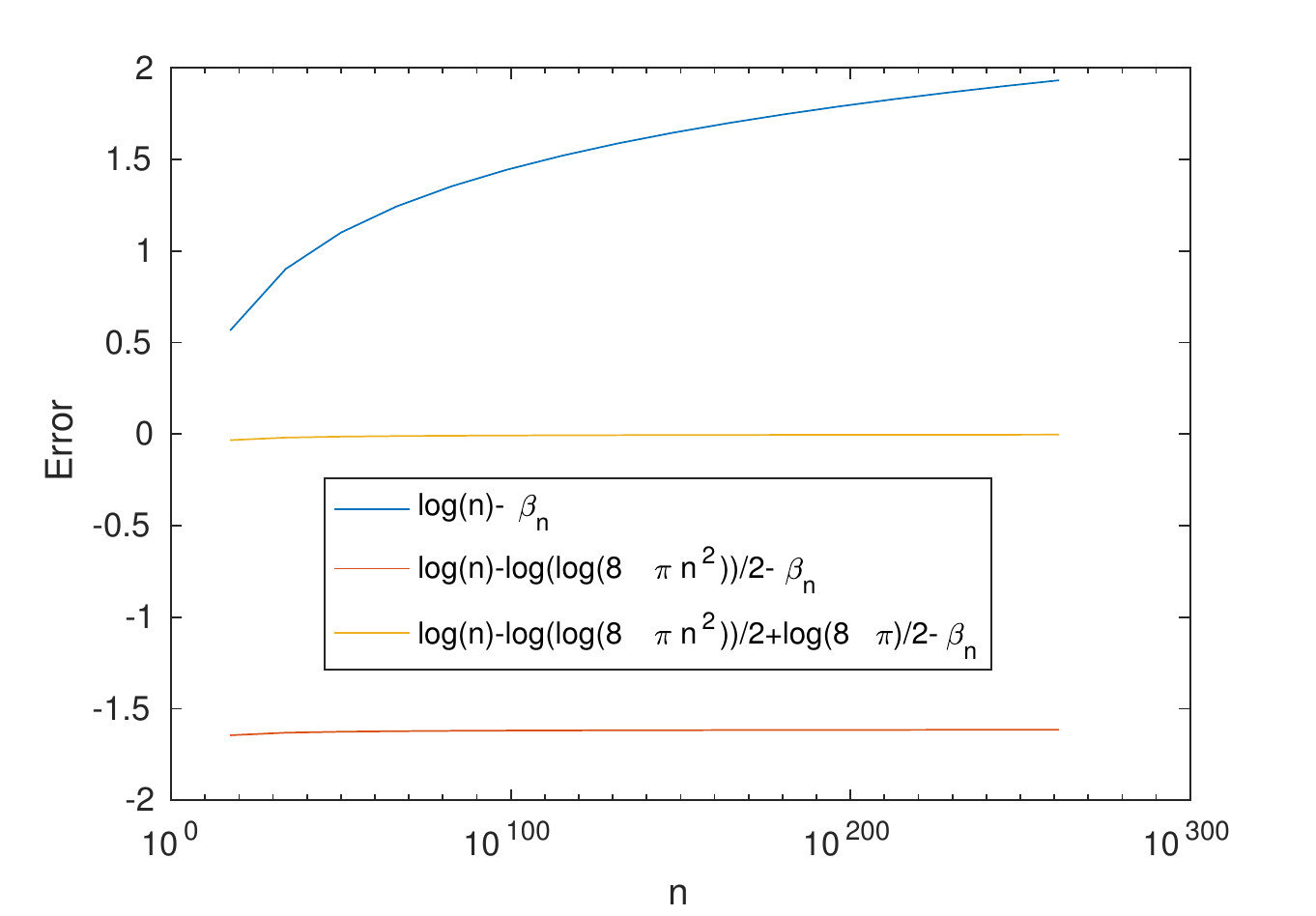} }%For arXiv
\subfloat{ \includegraphics[width = 0.49\textwidth]{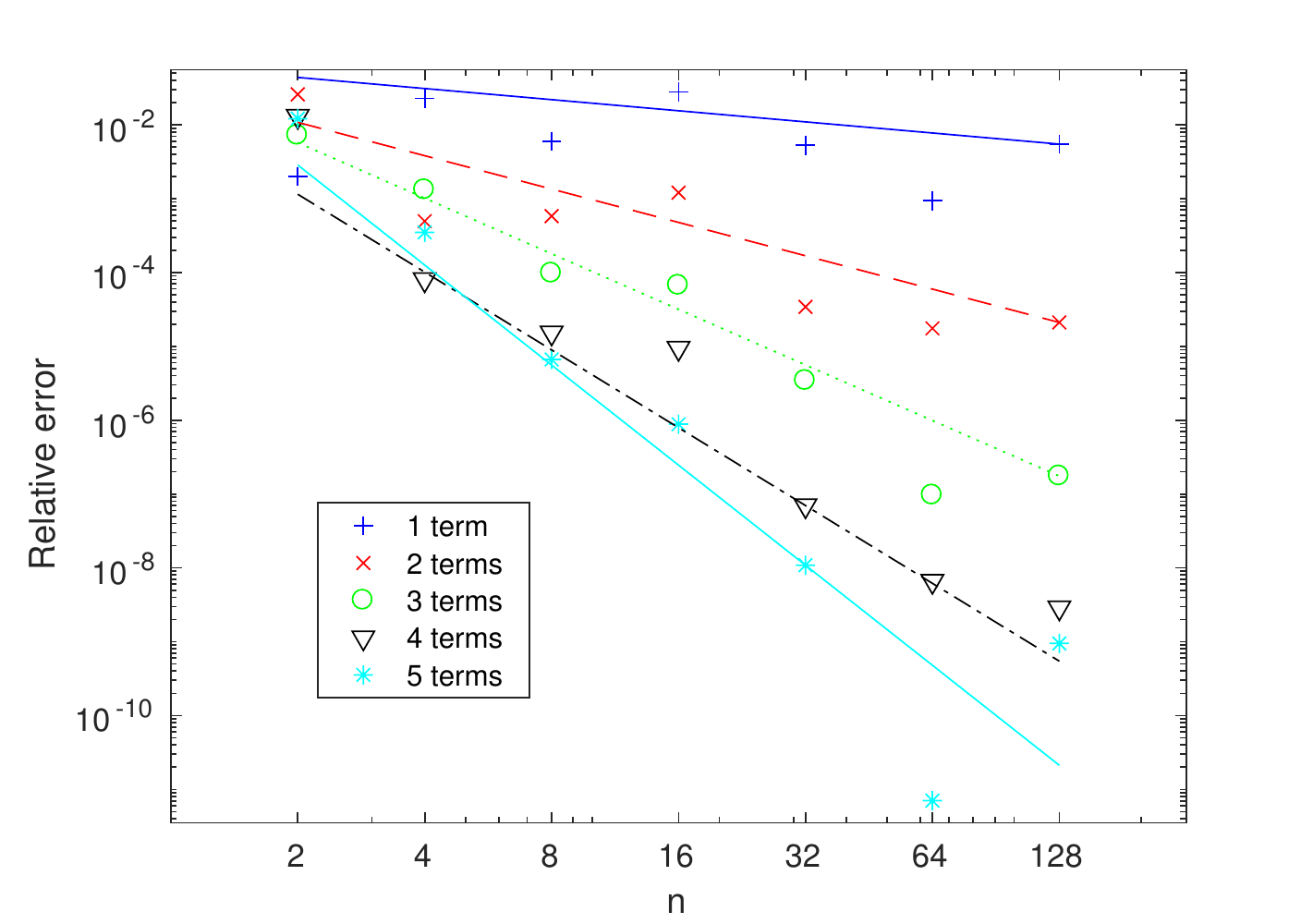} } %For arXiv
%\subfloat{ \includegraphics[width = 0.49\textwidth]{fig/errBetaExp} }
%\subfloat{ \includegraphics[width = 0.49\textwidth]{fig/bulk0p15Expnew} } 
\caption{Error of the asymptotic expansion of $\beta_n$ (left) and $p_{n}(0.6\beta_n)$ (right) for $w(x) = x^{-1/2}\exp(-\exp(x))$. }%\caption{Relative error of the asymptotic expansion of $b_{n-1}$ (left) and $p_{n}(0.6\beta_n)$ (right), both for $w(x) = x^{-1/2}\exp(-\exp(x))$. }
\label{Fexp}
\end{figure}

\subsection{General polynomial $Q(x)$ also used as a general function} \label{SnumerGen}

For this experiment, we recall \cref{Rgen,RgenCont}, which state that one can use the procedure for general functions also for polynomial $Q(x)$. \Cref{Fgen} provides a comparison of the accuracy obtained for $Q(x) = x^6 -2.1x^5 + 3x^3 -6x^2 +9$. The procedure for general polynomials would need about $m=6$ times more terms to achieve the same order in $n$ of the asymptotic expansion than the procedure for general functions. However, where the number of terms is too low in the left part of the figure, we appear to see a divergence. This is because the accuracies of the phase function $f_n(z)$ and the MRS number $\beta_n$ are too low to cancel out the exponential behaviour $\exp(n(V_n(z) + l_n)/2)$ %$\exp(n(V_n(z)/2 + l_n/2))$
in \cref{Epiboun} when calculated with too few terms. The reference recurrence coefficients were again computed using 26 digits, but we appear to see an $\mathcal{O}(n)$ error in the right panel of \cref{Fgen} at low errors. This means that we would need even more digits for this more difficult weight function if 
we would need to see the convergence for the expansions with more than three terms and all $n$. However, computing all recurrence coefficients is a very time consuming operation, while using asymptotics can be more accurate and orders of magnitude faster.

\begin{figure}[h]
\centering
\subfloat{ \includegraphics[width = 0.49\textwidth]{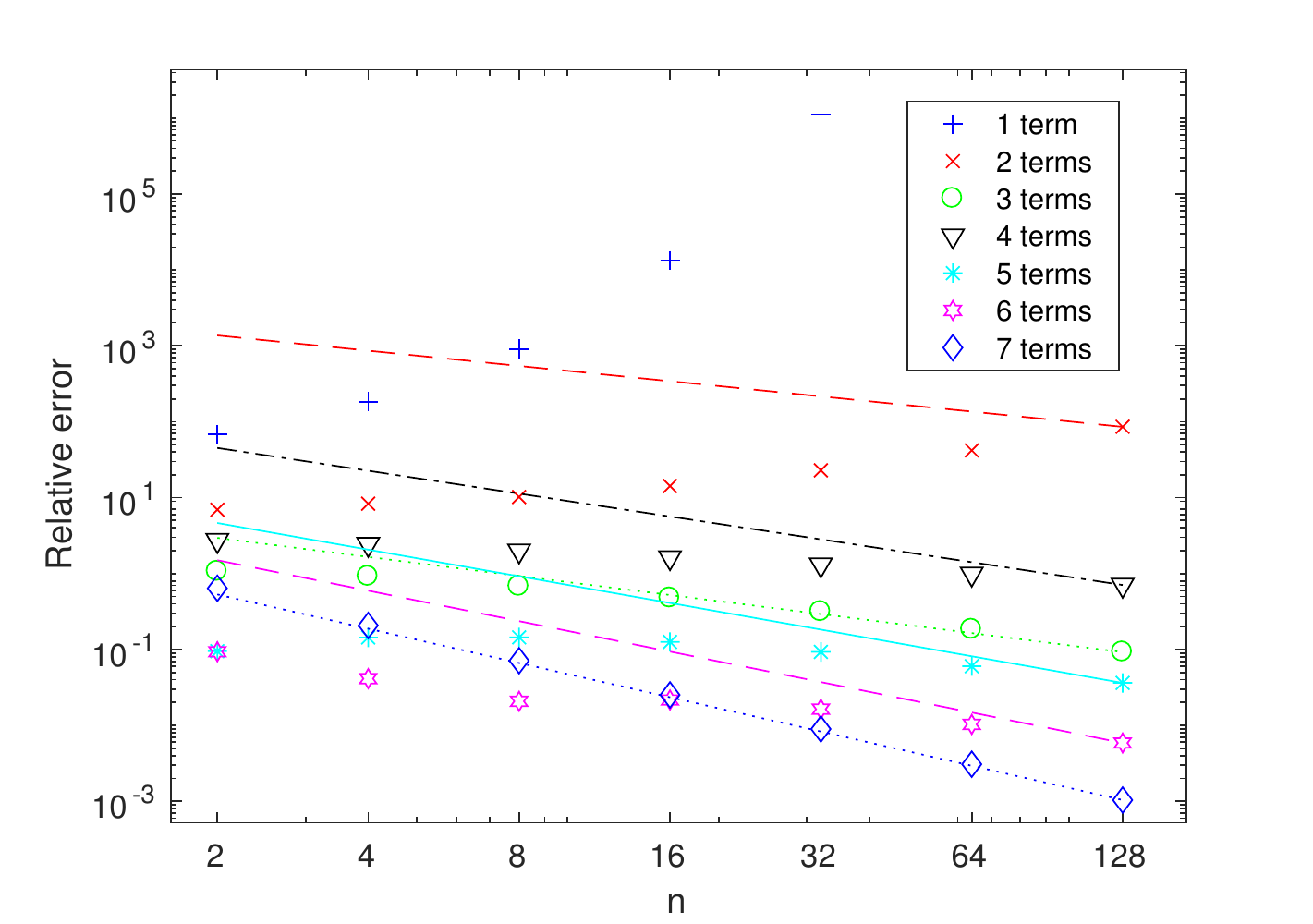} } %For arXiv
\subfloat{ \includegraphics[width = 0.49\textwidth]{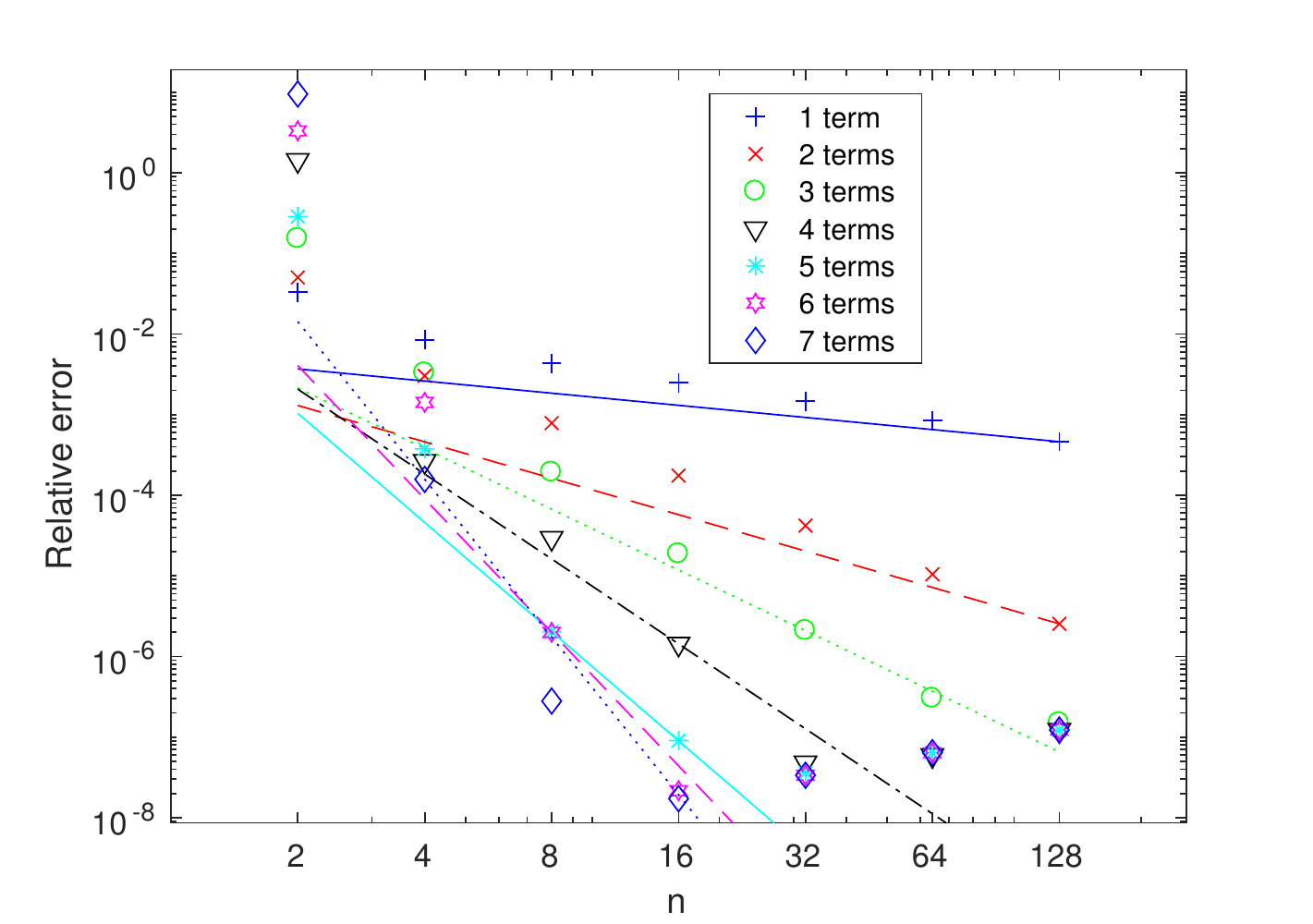} } %For arXiv
%\subfloat{ \includegraphics[width = 0.49\textwidth]{fig/genPoly} }
%\subfloat{ \includegraphics[width = 0.49\textwidth]{fig/genFct} } 
\caption{Relative error of the asymptotic expansion of $p_{n}((0.99+0.02i)\beta_n)$ in the right disk using the $Q$-matrices for $w(x) = x^{-1.1}\exp(-x^6 +2.1x^5 - 3x^3 +6x^2 -9)$ using the procedure for general polynomials $Q(x)$ (left) and general functions $Q(x)$ (right). Note that these are not directly comparable, as each term gives a $\mathcal{O}(n^{-T/m})$ error in the left part, while this is $\mathcal{O}(n^{-T})$ for the procedure for general polynomials.}
\label{Fgen}
\end{figure}

The procedure for general functions thus provides a higher accuracy (for the same number of terms), also because $\beta_n$ is computed up to an accuracy independent of $n$. \Cref{Tgen} shows that also the precomputations (computing the $U$ and $Q$ matrices as in \cref{S:higherorderterms,SexplL}) are faster using Matlab2016b on a $64$-bit laptop with $7.7$ GB memory and $4$ Intel(R) Core(TM) i7-3540M CPU's at $3.0$ Ghz. However, the mean time over the values of $n$ in \cref{Fgen} needed for evaluating the polynomial is much higher because the procedure for general functions computed double numerical integrals, as mentioned in \cref{SmrsNonpoly}. Thus, the procedure for general polynomial $Q(x)$ can be preferable when many evaluations of the polynomial are needed. The first row of \cref{Tgen} also shows that the time to evaluate the polynomial grows with the number of terms: the complexity is $\mathcal{O}(T^2)$ through \cref{ERQnonmon} and \cref{asympRn} if both indices $n$ and $k$ are proportional to $T$. The 
precomputations require $\mathcal{O}(T^5)$ operations for general polynomials due to the double summation for $g_{k,n}^l$ in \cref{SprecLnonmon}. For general functions however, that derivation of higher order terms has to be repeated for each value of $n$.

\begin{table}[h] 
\centering 
\hspace*{0.0cm} \begin{tabular}{l|cccc} 
Time (s) & Precomputations & $T=1$ & $T=4$ & $T=7$ \\ 
\hline 
General polynomials & 5.73e0   & 4.69e-3 & 6.95e-3 & 8.15e-3 \\ % 25 Nov
General functions   & 1.96e-1  & 2.24e-1 & 2.24e-1 & 2.24e-1
\end{tabular} 
\caption{Time required for different operations in seconds: mean time over seven values of $n$ of precomputations for seven terms in the expansion (only once for the procedure for general polynomials) and evaluating the asymptotic expansion for different numbers of terms $T$.}
\label{Tgen} 
\end{table}

\section*{Acknowledgments}
The authors gratefully acknowledge financial support from FWO (Fonds Wetenschappelijk Onderzoek, Research Foundation - Flanders, Belgium), through FWO research projects G.0617.10, G.0641.11 and G.A004.14. The authors would like to thank Alfredo Dea\~{n}o, Arno Kuijlaars, Alex Townsend, Walter Van Assche and Marcus Webb for useful discussions on the topic of this paper.

\begin{appendix}
\section{Explicit formulas for the first higher order terms} \label{APP:explicit}

The recursive computation of $R_k$ can give an arbitrary number of terms, but we provide the first few terms explicitly here for $z$ outside the two disks, ignoring the procedure for general polynomials. Expressions for $R^{\L/\R}(z)$ can straightforwardly be obtained by identifying $R_k^{\O}(z)$ from the following and using \cref{RHPforRk,asympRnmon}. We have:
\begin{align}
	R^{\O}(z) & = I + \frac{1}{n}\left(\frac{U_{1,1}^{\R} }{z-1}+\frac{U_{1,2}^{\R} }{(z-1)^2}+\frac{U_{1,1}^{\L}}{z} \right) + \frac{1}{n^2}\left(\frac{U_{2,1}^{\R} }{z-1}+\frac{U_{2,2}^{\R} }{(z-1)^2}+\frac{U_{2,3}^{\R} }{(z-1)^3}+\frac{U_{2,1}^{\L}}{z} \right)  \nonumber \\
	& + \frac{1}{n^3}\left(\frac{U_{3,1}^{\R} }{z-1} + \frac{U_{3,2}^{\R} }{(z-1)^2} + \frac{U_{3,3}^{\R} }{(z-1)^3} + \frac{U_{3,4}^{\R} }{(z-1)^4} + \frac{U_{3,5}^{\R} }{(z-1)^5} +\frac{U_{3,1}^{\L}}{z} + \frac{U_{3,2}^{\L}}{z^2} \right) +\mathcal{O}\left(\frac{1}{n^4}\right), \nonumber
\end{align}
with for general $Q(x)$
\begin{align}
	U_{1,1}^{\L} = & \frac{4\alpha^2-1}{2^{4}d_0 }\begin{pmatrix} 1 & 4^{-\alpha} i \\ 4^\alpha i & -1 \end{pmatrix}, \nonumber \\
	U_{1,1}^{\R} = & \frac{1}{2^{4} 3 c_0^2}\begin{pmatrix} -3(4\alpha^2c_0 - c_0 - c_1) & (12\alpha^2 c_0 + 24\alpha c_0 + 11c_0 - 3c_1)4^{-\alpha} i \\ (12 \alpha^2 c_0 - 24\alpha c_0 + 11c_0 - 3c_1)4^\alpha i & 3(4\alpha^2 c_0 - c_0 - c_1) \end{pmatrix}, \nonumber \\
	U_{1,2}^{\R} = & \frac{5}{2^{4} 3 c_0 }\begin{pmatrix} -1 & i4^{-\alpha} \\ 4^\alpha i & 1 \end{pmatrix}, \nonumber \\
	U_{2,1}^{\L} = & \frac{4\alpha^2-1}{2^{7} 3 c_0^2d_0^2 }\begin{pmatrix} \text{\foreignlanguage{russian}{ш}}(\alpha)  & \text{\foreignlanguage{russian}{ъ}}(\alpha)4^{-\alpha} i \\  -\text{\foreignlanguage{russian}{ъ}}(-\alpha)4^{\alpha} i  &  \text{\foreignlanguage{russian}{ш}}(-\alpha) \end{pmatrix}, \nonumber \\
	\text{\foreignlanguage{russian}{ш}}(b) = & 12 b^2 c_0 d_0 - 24 b c_0^2 + 12 b c_0 d_0 - c_0 d_0 - 3c_1 d_0, \\
	\text{\foreignlanguage{russian}{ъ}}(b) = & 12 b^2 c_0^2 + 12 b^2 c_0 d_0 - 24 b c_0^2 + 12 b c_0 d_0 - 27 c_0^2 - c_0 d_0 - 3c_1d_0, \\
	U_{2,1}^{\R} = & \frac{1}{2^{7} 3^{2} c_0^4 d_0 }\begin{pmatrix} -3\text{\foreignlanguage{russian}{л}}(\alpha) & \text{\foreignlanguage{russian}{б}}(-\alpha)4^{-\alpha} i \\ -\text{\foreignlanguage{russian}{б}}(\alpha)4^\alpha i & -3\text{\foreignlanguage{russian}{л}}(-\alpha) \end{pmatrix}, \nonumber \\
	\text{\foreignlanguage{russian}{л}}(b) = & 48 b^4 c_0^3 - 48 b^3c_0^3 - 96 b^3c_0 c_1 d_0 - 16 b^2 c_0^3 - 12 b^2 c_0^2 c_1 + 12 b c_0^3  \\
	&  + 24 b c_0 c_1 d_0 + 144 b c_1^2 d_0 - 120 b c_0 c_2 d_0 + c_0^3 + 3 c_0^2 c_1,  \\
	\text{\foreignlanguage{russian}{б}}(b) = & 144 b^4 c_0^3 + 144 b^4 c_0^2 d_0 - 144 b^3 c_0^3 - 384 b^3 c_0^2 d_0 + 288 b^3 c_0 c_1 d_0 - 48 b^2 c_0^3 - 36 b^2 c_0^2 c_1 + 264 b^2 c_0^2 d_0 - 936 b^2 c_0 c_1 d_0 \\
	&  + 36 b c_0^3 + 936 b c_0 c_1 d_0 - 432 b c_1^2 d_0 + 360 b c_0 c_2 d_0 + 3c_0^3 + 9c_0^2 c_1 - 23 c_0^2 d_0 - 282 c_0 c_1 d_0 - 441 c_1^2 d_0 - 360 c_0 c_2 d_0, \\
	U_{2,2}^{\R} = & \frac{1}{2^{7} 3^{2}c_0^3 d_0 }\begin{pmatrix} \text{\foreignlanguage{russian}{ф}}(-\alpha)  & \text{\foreignlanguage{russian}{ь}}(\alpha)4^{-\alpha} i \\ -\text{\foreignlanguage{russian}{ь}}(-\alpha)4^\alpha i & \text{\foreignlanguage{russian}{ф}}(\alpha)  \end{pmatrix}, \nonumber \\
	\text{\foreignlanguage{russian}{ф}}(b) = & 240 b^3 c_0 d_0 - 60 b^2 c_0^2 + 60 b^2 c_0 d_0 - 36 b c_0 d_0 -468 b c_1 d_0 + 15c_0^2 - 28 c_0 d_0 + 21 c_1 d_0, \\
	\text{\foreignlanguage{russian}{ь}}(b) = & 240 b^3 c_0 d_0 + 60 b^2 c_0^2 + 780 b^2 c_0 d_0 + 804 b c_0 d_0 -468 b c_1 d_0 - 15 c_0^2 + 259 c_0 d_0 - 483 c_1 d_0, \\
	U_{2,3}^{\R} = & \frac{35}{2^{7} 3^{2} c_0^2 }\begin{pmatrix} -12\alpha - 1 & 3(\alpha + 1)4^{1-\alpha} i \\ 3(\alpha - 1)4^{\alpha + 1} i & 12\alpha - 1 \end{pmatrix}. \nonumber
\end{align}
This agrees with results by Vanlessen: \cite[(4.11)]{Vanlessen} equals $U_{1,1}^{\R}+U_{1,1}^{\L}$ and \cite[(4.12)]{Vanlessen} equals $\left.U_{1,1}^{\R}+U_{1,2}^{\R}\right|_{2,1}$. For $w(x) = x^\alpha \exp(-x)$, one can use $c_0 = 4 = d_0$, $c_1 = 0 = c_2$ and the next higher order term is given by
\begin{align}
	U_{3,1}^{\L} = & \frac{(4\alpha^2-1)}{2^{17} 3^{2} }\begin{pmatrix} 
\text{\foreignlanguage{russian}{я}}(\alpha) & \text{\foreignlanguage{russian}{и}}(\alpha)4^{-\alpha} i \\ \text{\foreignlanguage{russian}{и}}(-\alpha) 4^\alpha i & -\text{\foreignlanguage{russian}{я}}(-\alpha) \end{pmatrix}, \nonumber \\
	\text{\foreignlanguage{russian}{я}}(b) = & 288 b^4 - 960 b^3 + 444 b^2 + 768 b - 305, \\
	\text{\foreignlanguage{russian}{и}}(b) = & 768 b^4 - 1824 b^3 - 1284 b^2 + 2712b + 1153, \\
	U_{3,2}^{\L} = & \frac{(4\alpha^2-1)(4\alpha^2 - 9)(4\alpha^2 - 25)}{2^{17} 3 }\begin{pmatrix} -1 & -i4^{-\alpha} \\ -4^\alpha i & 1 \end{pmatrix}, \nonumber \\
	U_{3,1}^{\R} = & \frac{1}{2^{17} 3^{4} 5}\begin{pmatrix} -\text{\foreignlanguage{russian}{щ}}(\alpha) & \text{\foreignlanguage{russian}{д}}(\alpha)4^{-\alpha} i \\ \text{\foreignlanguage{russian}{д}}(-\alpha)4^\alpha i & \text{\foreignlanguage{russian}{щ}}(-\alpha) \end{pmatrix}, \nonumber \\
	\text{\foreignlanguage{russian}{щ}}(b) = & 45(288 b^4 - 960 b^3 + 444 b^2 + 768 b - 305)(4\alpha^2 - 1), \\
	\text{\foreignlanguage{russian}{д}}(b) = & 138240 b^6 + 51840 b^5 - 287280 b^4 - 109440 b^3 + 103320 b^2 + 29880 b - 11603, \\
	U_{3,2}^{\R} = & \frac{1}{2^{16} 3^{4} 5 }\begin{pmatrix} -\text{\foreignlanguage{russian}{ц}}(-\alpha) & 2i4^{-\alpha} \text{\foreignlanguage{russian}{ы}}(\alpha) \\ 2i4^\alpha \text{\foreignlanguage{russian}{ы}}(-\alpha) & \text{\foreignlanguage{russian}{ц}}(\alpha) \end{pmatrix}, \nonumber \\
	\text{\foreignlanguage{russian}{ц}}(b) = & 4320 b^6 - 51840 b^5 - 64800 b^4 + 33120 b^3 + 13590 b^2 - 5760 b + 389, \\
	\text{\foreignlanguage{russian}{ы}}(b) = & 2160 b^6 + 56160 b^5 + 156600 b^4 + 119520 b^3 + 20655 b^2 - 7470 b - 1109, \\
	U_{3,3}^{\R} = & \frac{1}{2^{17} 3^{4} 5}\begin{pmatrix} -\text{\foreignlanguage{russian}{з}}(-\alpha) & 3i4^{-\alpha} \text{\foreignlanguage{russian}{ю}}(\alpha) \\ 3i4^\alpha\text{\foreignlanguage{russian}{ю}}(-\alpha) & \text{\foreignlanguage{russian}{з}}(\alpha) \end{pmatrix}, \nonumber \\
	\text{\foreignlanguage{russian}{з}}(b) = & 226800 b^4 - 100800 b^3 + 78120 b^2 - 19633, \\
	\text{\foreignlanguage{russian}{ю}}(b) = & 75600 b^4 + 403200 b^3 + 626640 b^2 + 434280 b + 114089, \\
	U_{3,4}^{\R} = & \frac{1}{2^{16} 3^{4} }\begin{pmatrix} -90090\alpha^2 - 12012 & 1001(90\alpha^2 + 180\alpha + 107)4^{-\alpha} i \\ 1001(90\alpha^2 - 180\alpha + 107)4^{\alpha} i & 90090\alpha^2 + 12012 \end{pmatrix}, \nonumber \\
	U_{3,5}^{\R} = & \frac{5\times 7\times 11 \times 13 \times 17}{2^{17} 3^{4} }\begin{pmatrix} -1 & i4^{-\alpha} \\ i 4^\alpha & 1 \end{pmatrix}. \nonumber 
\end{align}

\end{appendix}

\bibliographystyle{abbrv}
%\bibliography{refv3}
%\bibliography{refv4}
\bibliography{laguerre10.bbl}%For arXiv

\begin{thebibliography}{10}

\bibitem{BogaertIterationFree}
I.~Bogaert.
\newblock {Iteration-Free Computation of Gauss--Legendre Quadrature nodes and
  weights}.
\newblock {\em SIAM J. Sci. Comput.}, 36(3):A1008--A1026, 2014.

\bibitem{bogaert}
I.~Bogaert, B.~Michiels, and J.~Fostier.
\newblock {$\mathcal{O}(1)$ computation of {L}egendre polynomials and
  {G}auss-{L}egendre nodes and weights for parallel computing}.
\newblock {\em SIAM J. Sci. Comput.}, 34(3):C83--C101, 2012.

\bibitem{varying}
C.~Bosbach and W.~Gawronski.
\newblock Strong asymptotics for {L}aguerre polynomials with varying weights.
\newblock {\em J. Comput. Appl. Math.}, 99:77--89, 1998.

\bibitem{bremer}
J.~Bremer.
\newblock On the numerical calculation of the roots of special functions
  satisfying second order ordinary differential equations.
\newblock {\em SIAM J. Sc. Comput.}
\newblock To appear.

\bibitem{alfLag}
A.~Dea\~{n}o, E.~J. Huertas, and F.~Marcell\'an.
\newblock Strong and ratio asymptotics for {L}aguerre polynomials revisited.
\newblock {\em J. Math. Anal. Appl.}, 403:477--486, 2013.

\bibitem{jacobi}
A.~Dea\~{n}o, D.~Huybrechs, and P.~Opsomer.
\newblock {Construction and implementation of asymptotic expansions for
  Jacobi-type orthogonal polynomials}.
\newblock {\em Adv. Comput. Math.}, 42(4):791--822, 2016.

\bibitem{DKMVZ}
P.~Deift, T.~Kriecherbauaer, K.~T.-R. McLauglin, S.~Venakides, and X.~Zhou.
\newblock Strong asymptotics of orthogonal polynomials with respect to
  exponential weights.
\newblock {\em Comm. Pure Appl. Math.}, 52(12):1491--1552, 1999.

\bibitem{DZ}
P.~Deift and X.~Zhou.
\newblock {A Steepest Descent Method for Oscillatory Riemann--Hilbert
  Problems}.
\newblock {\em Bull. Amer. Math. Soc.}, 26(1):119--124, 1992.

\bibitem{dzsd}
P.~Deift and X.~Zhou.
\newblock {A steepest descent method for oscillatory Riemann--Hilbert problems.
  Asymptotics for the MKdV equation}.
\newblock {\em Ann. Math.}, 137:295--368, 1993.

\bibitem{fik}
A.~Fokas, A.~Its, and A.~Kitaev.
\newblock The isomonodromy approach to matrix models in 2d quantum gravity.
\newblock {\em Comm. Math. Phys.}, 147:395--430, 1992.

\bibitem{gil2016}
A.~Gil, J.~Segura, and N.~Temme.
\newblock Efficient computation of {L}aguerre polynomials.
\newblock {\em Computer Physics Communications}, 210:124--131, 2017.

\bibitem{glaser}
A.~Glaser, X.~Liu, and V.~Rokhlin.
\newblock {A fast algorithm for the calculation of the roots of special
  functions}.
\newblock {\em SIAM J. Sci. Comput.}, 29(4):1420--1438, 2007.

\bibitem{Gosper}
R.~W. Gosper.
\newblock {Decision procedure for indefinite hypergeometric summation}.
\newblock {\em Proc. Natl. Acad. Sci. USA}, 75(1):40--42, 1978.

\bibitem{HT}
N.~Hale and A.~Townsend.
\newblock Fast and accurate computation of {G}auss--{L}egendre and
  {G}auss--{J}acobi quadrature nodes and weights.
\newblock {\em SIAM J. Sci. Comput.}, 35:A652--A672, 2013.

\bibitem{KuijLect}
A.~Kuijlaars.
\newblock {\em Orthogonal polynomials and Special functions}, volume 1817 of
  {\em Lecture Notes in Mathematics}, chapter Riemann-Hilbert analysis for
  orthogonal polynomials, pages 167--210.
\newblock Springer-Verlag, New York, 2003.
\newblock Editors: E. Koelink and W. Van Assche.

\bibitem{KMcLVAV}
A.~B.~J. Kuijlaars, K.~T.-R. McLaughlin, W.~{Van Assche}, and M.~Vanlessen.
\newblock {The {R}iemann-{H}ilbert approach to strong asymptotics of orthogonal
  polynomials on $[-1,1]$}.
\newblock {\em Adv. Math.}, 188:337--398, 2004.

\bibitem{Levin2001weights}
E.~Levin and D.~Lubinsky.
\newblock {\em Orthogonal Polynomials for Exponential Weights}.
\newblock Springer, New York, 2001.

\bibitem{temme2004}
J.~L. L{\'o}pez and N.~M. Temme.
\newblock Convergent asymptotic expansions of {C}harlier, {L}aguerre and
  {J}acobi polynomials.
\newblock {\em {Proceedings of the Royal Society of Edinburgh Section A:
  Mathematics}}, 134:537--555, 2004.

\bibitem{DLMF}
NIST.
\newblock {NIST Digital Library of Mathematical Functions}.
\newblock http://dlmf.nist.gov/, 2016.
\newblock Online companion to \cite{Olver:2010:NHMF}.

\bibitem{Olver:2010:NHMF}
F.~W.~J. Olver, D.~W. Lozier, R.~F. Boisvert, and C.~W. Clark, editors.
\newblock {\em {NIST Handbook of Mathematical Functions}}.
\newblock Cambridge University Press, New York, NY, 2010.
\newblock Print companion to \cite{DLMF}.

\bibitem{ninesLag}
P.~Opsomer.
\newblock Asymptotic expansions of generalized {L}aguerre polynomials.
\newblock http://nines.cs.kuleuven.be/software/LAGUERRE, 2016.

\bibitem{Szego}
G.~Szeg\H{o}.
\newblock {\em {Orthogonal Polynomials: American Mathematical Society
  Colloquium publications Volume XXIII}}.
\newblock American Mathematical Society, Providence, Rhode Island, 3 edition,
  1967.

\bibitem{temme1990}
N.~M. Temme.
\newblock Asymptotic estimates for {L}aguerre polynomials.
\newblock {\em Journal of {A}pplied {M}athematics and {P}hysics {(ZAMP)}},
  41:114--126, 1990.

\bibitem{chebfun}
{The University of Oxford} and the Chebfun~Developers.
\newblock {Chebfun—numerical computing with functions}.
\newblock http://www.chebfun.org/, 2016.

\bibitem{FastGaussQuadr}
A.~Townsend.
\newblock {FastGaussQuadrature}.
\newblock https://github.com/ajt60gaibb/FastGaussQuadrature.jl, 2016.

\bibitem{TTOGauss}
A.~Townsend, T.~Trogdon, and S.~Olver.
\newblock {Fast computation of Gauss quadrature nodes and weights on the whole
  real line}.
\newblock {\em IMA J. Numer. Anal.}, 2015.

\bibitem{trefBau}
L.~N. Trefethen and D.~Bau.
\newblock {\em Numerical linear algebra}.
\newblock Society for Industrial and Applied Mathematics, Philadelphia, 1997.

\bibitem{Vanlessen}
M.~Vanlessen.
\newblock Strong asymptotics of {L}aguerre-type orthogonal polynomials and
  applications in random matrix theory.
\newblock {\em Constr. Approx.}, 25:125--175, 2007.

\bibitem{ZhaoPartition}
Y.~Zhao, L.~Cao, and D.~Dai.
\newblock Asymptotics of the partition function of a laguerre-type random
  matrix model.
\newblock {\em J. Approx. Theory}, 178:64--90, 2014.

\end{thebibliography}

\end{document}